\documentclass{amsart}
\usepackage{amsmath}
\usepackage{amssymb}
\usepackage{amsthm}
\usepackage{epsfig}
\usepackage{amsfonts}
\usepackage{amscd}
\usepackage{mathrsfs}
\usepackage{enumerate}
\usepackage{latexsym}

\newtheorem{theorem}{Theorem}[section]
\newtheorem{lemma}[theorem]{Lemma}
\newtheorem{proposition}[theorem]{Proposition}
\newtheorem{corollary}[theorem]{Corollary}
\newtheorem{question}[theorem]{Question}

\theoremstyle{definition}
\newtheorem{definition}[theorem]{Definition}
\newtheorem{example}[theorem]{Example}

\theoremstyle{remark}

\numberwithin{equation}{section}

\begin{document}

\title[One-point metrizable extensions]{On one-point metrizable extensions of locally compact metrizable spaces}

\author{M.R. Koushesh}
\address{Department of Mathematics, University of Manitoba, Winnipeg, Manitoba  R3T 2N2, Canada}
\email{umkoushe@cc.umanitoba.ca}



\begin{abstract}
For a non-compact metrizable space $X$, let  ${\mathcal E}(X)$ be the set of all
 one-point metrizable extensions of $X$, and when $X$ is locally compact, let ${\mathcal E}_K(X)$ denote the set of all locally compact elements of
 ${\mathcal E}(X)$ and $\lambda: {\mathcal E}(X)\rightarrow{\mathcal Z}(\beta X\backslash X)$ be the
 order-anti-isomorphism (onto its image)
  defined in: [HJW] M. Henriksen, L. Janos and R.G. Woods,  Properties of one-point completions of a non-compact metrizable space,
Comment. Math. Univ. Carolinae 46 (2005), 105-123. By
definition $\lambda(Y)= \bigcap_{n<\omega}\mbox {cl}_{\beta
X}(U_n\cap X)\backslash X$, where $Y=X\cup\{p\}\in {\mathcal E}(X)$
and $\{U_n\}_{n<\omega}$ is an open base at $p$ in $Y$. Answering
the question of [HJW], we characterize the elements of the image of $\lambda$ as exactly those non-empty zero-sets of
$\beta X$ which miss $X$, and the elements of the image of ${\mathcal
E}_K(X)$ under $\lambda $, as those which are moreover clopen in $\beta X\backslash
X$. We then study the relation between ${\mathcal E}(X)$ and
${\mathcal E}_K(X)$ and their order structures, and introduce a
subset ${\mathcal E}_S(X)$ of ${\mathcal E}(X)$. We conclude with
some theorems on the cardinality of the sets ${\mathcal E}(X)$ and
${\mathcal E}_K(X)$,
 and some open questions.
\end{abstract}

\maketitle

\section{Introduction}

If a Tychonoff space $Y$ contains a space $X$ as a dense subspace,
then $Y$ is called an {\em extension} of $X$. Two extensions $Y_1$
and $Y_2$ of $X$ are said to be {\em equivalent} if there exists a
homeomorphism of $Y_1$ onto $Y_2$ which keeps $X$ pointwise fixed.
This is an equivalence relation which partitions
 the set of Tychonoff extensions of $X$ into equivalence classes. We
identify these equivalence classes with individuals whenever no
confusion arises. For two Tychonoff extensions $Y_1$ and $Y_2$ of
the space $X$, we let $Y_1\leq Y_2$ if there exists a continuous
function from $Y_2$ into $Y_1$ which keeps $X$ pointwise fixed. This
in fact, defines a partial order on the set of all Tychonoff
extensions of the space $X$. We refer the reader to Section \ref{GGH} of
[14] for a detailed discussion on this  subject.

In this paper we are only concerned with those extensions $Y$ of
a space $X$ for which $Y\backslash X$ is a singleton. Such kind of
extensions are called {\em one-point extensions}. One-point
extensions are studied extensively. For some results as well as
some bibliographies on the subject see [11] and [12]. The
present work is based on  [9], in which the authors
studied the one-point extensions of locally compact metrizable
spaces. Their work was in turn, motivated by Bel'nov's studies of
the set of all metric extensions of metrizable spaces. In [9],
for a locally compact separable metrizable space $X$, the authors
investigated the relation between the order structure of the set
of all one-point metrizable extensions of $X$ and the topology of
the space $\beta X\backslash X$ (as usual $\beta X$ is the
Stone-\v{C}ech compactification of $X$). One of the earliest
results of this sort is due to  Magill [13] who proved that if
${\mathcal K}(X)$ and ${\mathcal K}(Y)$ denote the set of all
compactifictions of locally compact spaces $X$ and $Y$,
respectively, then ${\mathcal K}(X)$ and ${\mathcal K}(Y)$ are
order-isomorphic if and only if $\beta X\backslash X$
and $\beta Y\backslash Y$ are homeomorphic. Part of this paper
is devoted to the relation between the order structure of certain
subsets of one-point metrizable extensions of a locally compact
non-separable metrizable space $X$ and the topology of certain
subspaces of $\beta X\backslash X$.  We now review some of the
notations and give a brief list of the main results of [9]
that will be used in the sequel. This at the same time makes our
work self-contained.

The letters $\bf I$ and $\bf R$ denote the closed unit real interval
and the real line, respectively. For spaces $X$ and $Y$, $C(X, Y)$
denotes the set of all continuous functions from $X$ to $Y$. For
$f\in C(X, \bf {R}) $, we denote by $Z(f)$ and $Coz(f)$, the
zero-set and the cozero-set of $f$, respectively. The support of $f$
is the set $\mbox {cl}_X(Coz(f))$ and is denoted by $\mbox
{supp}(f)$. By ${\mathcal Z} (X)$ we mean the set of all zero-sets
of $f\in C(X, \bf {R}) $. A subset of a space $X$ is called {\em clopen in} $X$, if it is simultaneously closed  and open in $X$. We denote by  ${\mathcal B} (X)$ the set of all clopen subsets of a space $X$. The weight of $X$ is denoted by $w(X)$. The
letters $\omega$ and $\omega_1$ denote the first infinite countable
and the first uncountable ordinal numbers. We denote by $\aleph_0$
and $\aleph_1$ the cardinalities of $\omega$ and $\omega_1$,
respectively.  The symbol [CH] denotes the Continuum Hypothesis, and
whenever it appears at the beginning of the statement of a theorem,
it indicates that the Continuum Hypothesis is being assumed in  that
theorem. The two symbols $\bigvee$ and $\bigwedge$ are used to
denote the  least upper bound and the greatest lower bound,
respectively. If $P$ and $Q$ are partially ordered sets, a function $f:P\rightarrow Q$ is called an {\em order-homomorphism} ({\em order-anti-homomorphism}, respectively) if $f(a)\leq f(b)$ ($f(a)\geq f(b)$, respectively) whenever $a\leq  b$. The function $f$ is called  an {\em order-isomorphism} ({\em order-anti-isomorphism}, respectively) if it is moreover bijective and $f^{-1}:Q\rightarrow P$ is also an order-homomorphism (order-anti-homomorphism, respectively). The partially ordered sets $P$ and $Q$ are called {\em order-isomorphic} ({\em order-anti-isomorphic}, respectively) if there is an order-isomorphism (order-anti-isomorphism, respectively) between them.

Let $X$ be a non-compact metrizable space. We denote
by ${\mathcal E} (X)$ the set of all one-point metrizable extensions of
$X$. A sequence ${\mathcal U}=\{U_n\}_{n<\omega}$ of non-empty open
subsets of $X$ is called a {\em regular sequence of open sets in
$X$}, if for each $n<\omega$, $\mbox {cl}_X U_{n+1}\subseteq U_n$.
If moreover $\bigcap_{n<\omega}U_n=\emptyset$, we call ${\mathcal
U}$ an {\em extension trace in $X$} (Definition \ref{LYFG} of [9]). Every
extension trace $\{U_n\}_{n<\omega}$ in $X$ generates a one-point metrizable
extension of $X$. In fact if we let $Y=X\cup\{p\}$, where $p\notin
X$, and define
\[{\mathcal O}_Y={\mathcal O}_X\cup\big\{V\cup\{p\}: V
\mbox { is open in } X \mbox { and } V\supseteq U_n \mbox {, for
some } n<\omega \big\}\]
where ${\mathcal O}_X$ is the set of open
subsets of $X$, then $(Y, {\mathcal O}_Y)$ constitutes a one-point
metrizable extension of $X$ (see Theorem 2 of [1] or Theorem \ref{FGGF} of [2]). Conversely, if for $Y=X\cup\{p\}\in
{\mathcal E} (X)$, we let for each $n<\omega$, $U_n=B(p, 1/n )\cap
X$, then $\{U_n\}_{n<\omega}$ is an extension trace in $X$ which
generates $Y$. If $Y_{\mathcal U}=X\cup\{p\}$ is the one-point
metrizable  extension generated by the extension trace ${\mathcal
U}=\{U_n\}_{n<\omega}$, then the set $\{U_n\cup\{p\}\}_{n<\omega}$
forms an open base at $p$ in $Y_{\mathcal U}$.  For two extension
traces ${\mathcal U}=\{U_n\}_{n<\omega}$ and ${\mathcal
V}=\{V_n\}_{n<\omega}$ in $X$, we say that {\em ${\mathcal U}$ is
finer than ${\mathcal V}$}, if for each $n<\omega$, there exists a
$k_n<\omega$ such that $U_{k_n}\subseteq V_n$. For  extension traces
${\mathcal U}$ and ${\mathcal V}$ in $X$,  $Y_{\mathcal U}\geq
Y_{\mathcal V}$ if and only if  ${\mathcal U}$ is finer than
${\mathcal V}$ (Theorem \ref{LPPG} of [9]). For $A \subseteq X$, $A^* $
is $ (\mbox {cl}_{\beta X} A)\backslash X$, in particular $X^*=\beta
X\backslash X$.  If $X$ is moreover locally compact, we let
$\lambda:{\mathcal E} (X)\rightarrow{\mathcal Z} (X^*)$ be defined
by $\lambda(Y)= \bigcap_{n<\omega}U_n^*$, where $\{U_n\}_{n<\omega}$
is an extension trace in $X$ which generates $Y$. The function
$\lambda$ is well-defined, and it is an order-anti-isomorphism onto
its image (Theorem \ref{GGH}0 of [9]). If $X$ is moreover separable,
then $\lambda({\mathcal E} (X))={\mathcal Z}
(X^*)\backslash\{\emptyset\}$, and therefore in this case,
${\mathcal E} (X)$ and ${\mathcal Z}
(X^*)\backslash\{\emptyset\}$ are order-anti-isomorphic. Thus for
locally compact non-compact separable metrizable spaces $X$ and $Y$,
${\mathcal E} (X)$ and ${\mathcal E} (Y)$ are order-isomorphic if and only
if
$X^*$ and $Y^*$ are homeomorphic (Theorem \ref{TTG} of [9]). If
${\mathcal E}_K (X)$ is the locally compact elements of ${\mathcal
E} (X)$, then when $X$ is separable, $\lambda({\mathcal E}_K (X))$
is the set of all non-empty clopen sets of $X^*$ (Theorems \ref{TUUG} and
\ref{WG} of [9]). Thus when $X$ and $Y$ are locally compact non-compact
separable metrizable spaces whose Stone-\v{C}ech remainders are
zero-dimensional, then ${\mathcal E}_K (X)$ and ${\mathcal E}_K (Y)$
are order-isomorphic if and only if $X^*$ and $Y^*$ are homeomorphic
(Theorem \ref{EWG} of [9]). If $X$ is non-separable, we let $\sigma X$
denote the subset of $\beta X$ consisting of those points of $\beta
X$ which are in the closure in $\beta X$ of some $\sigma$-compact
subset of $X$. We make use of the following theorem  in a number of
occasions (see \ref{FGEF}.F of [7]).

\begin{theorem}[Alexandroff]\label{IUUG}
If $X$ is a locally
compact non-separable metrizable space then it can be written as
\[ X=\bigoplus_{i\in I}X_i, \mbox { where each $X_i$ is a non-compact separable
subspace.}\]
\end{theorem}

Now using the above notations, for a locally compact
non-separable metrizable space $X$, we have
\[\sigma X=\bigcup\Big\{\mbox{cl}_{\beta X}\Big(\bigcup_{i\in J} X_i\Big):
J\subseteq I \mbox { is countable}\Big\}.\]
Clearly $\sigma X$ is an open subset of $\beta X$. If
$c\sigma X =\beta X\backslash \sigma X$, for a locally compact non-separable metrizable space $X$, if $Z\in\lambda({\mathcal E} (X))$,
then $\mbox {int}_{c\sigma X}(Z\backslash \sigma X)=\emptyset$
(Theorem 6.6 of [9]). When $X$ is an uncountable discrete space
the converse also holds, i.e.,  for $\emptyset\neq Z\in {\mathcal Z}
(X^*)$, $Z\in\lambda({\mathcal E} (X))$ if and only if $\mbox
{int}_{c\sigma X}(Z\backslash \sigma X)=\emptyset$ (Theorem 6.8 of
[9]). When $X$ is locally compact, to every $\emptyset\neq Z\in
{\mathcal Z} (X^*)$, there corresponds a regular sequence of open
sets $\{U_n\}_{n<\omega}$ in $X$ for which $Z=
\bigcap_{n<\omega}U_n^*$ (Corollary \ref{FGGF} of [9]). Theorem 6.7 of
[9] characterizes the elements of $\lambda({\mathcal E} (X))$ in
terms of the regular sequences of open sets generating them as
follows. Suppose that  $X$ is a locally compact non-separable
metrizable space and let $\emptyset\neq Z\in {\mathcal Z} (X^*)$.
Then $Z\in \lambda({\mathcal E} (X))$ if and only if there does not
exists $S\in {\mathcal Z} (X)$ such that $\emptyset\neq
S^*\backslash \sigma X\subseteq Z\backslash \sigma X$ if and only if
 $\mbox {cl}_{\beta X}(\bigcap_{n<\omega}U_n)\subseteq\sigma X$, where
$\{U_n\}_{n<\omega}$ is a regular sequence of open sets in $X$  for
which $Z= \bigcap_{n<\omega}U_n^*$, if and only if
 $\bigcap_{n<\omega}U_n$ is $\sigma$-compact (with
$\{U_n\}_{n<\omega}$ as in the previous condition).

We also make
use of the following well known result. If $Z_1,\ldots, Z_n$  are
zero-sets in $X$, then $\mbox {cl}_{\beta X}(\bigcap_{i=1}^{n}
Z_i)=\bigcap_{i=1}^{n}\mbox {cl}_{\beta X} Z_i$. For other undefined
terms and notations we refer the reader to the texts of [7], [8] and [14].

\section{Characterization of the image of $\lambda$}

In Theorem 6.8 of [9], the authors  characterized the image of
$\lambda$ for an uncountable discrete space $X$, as the set of all
non-empty zero-sets of $X^*$ such that $\mbox {int}_{c\sigma X}
(Z\backslash\sigma X)=\emptyset$. They also asked whether this
characterization can be generalized to the case when $X$ is any
locally compact non-separable metrizable space. In the following
theorem we answer this question by characterizing those spaces $X$
for which the above characterization of the image of $\lambda$
holds.

\begin{theorem}\label{LKG}
Let $X$ be a locally compact non-separable metrizable space and let  $X=\bigoplus _{i\in I} X_i
$, where each $X_i$ is a separable non-compact subspace. Then we
have
\begin{itemize}
\item[\rm(1)] If at most countably many of
the  $X_i$'s are non-discrete, then for a  $\emptyset\neq Z\in
{\mathcal Z}(X^*)$, we have  $ Z\in \lambda({\mathcal E}(X))$ if and
only if $\mbox {\em int}_{c\sigma X} (Z\backslash\sigma
X)=\emptyset$;
\item[\rm(2)] If uncountably many of the  $X_i$'s are non-discrete, then there
exists a  $\emptyset\neq Z\in {\mathcal Z}(X^*)$ such that $\mbox
{\em int}_{c\sigma X} (Z\backslash\sigma X)=\emptyset$, but $Z\notin
\lambda({\mathcal E}(X))$.
\end{itemize}
\end{theorem}

\begin{proof}
1) This follows by a modification of the proof given in Theorem  6.8 of [9].

2)  Let $L= \{ i \in I : X_i \mbox { is  not a discrete space}\}$.
Suppose that $\{ L_n\}_{n <\omega }$ is a partition of $L$ into
mutually disjoint uncountable subsets. For convenience, let the
metric on $X$ be chosen to be bounded by 1, and such that
$d(x,y)=1$, if $x$ and $y$ do not belong to the same factor $X_i$.
Since for each $i\in L$, $X_i$ is non-discrete, there exists a
non-trivial convergent sequence $\{x^i_n \}_{n< \omega }$ in $X_i$.
Let, for each $i\in L$, $x_i$ denote the limit of the sequence
$\{x^i_n \}_{n< \omega }$ and assume that  $x_i \notin \{x_n^i\}_{
n<\omega}$. We define a sequence $\{U_n\}_{n<\omega}$  by letting
\[U_n =\bigcup\Big\{B_d\Big(x_i, \frac{1}{n+k}\Big): k<\omega \mbox{ and } i\in
L_k\Big\}.\]
Clearly $\{U_n \}_{n< \omega }$ is a regular sequence of
open sets in $X$. By  Lemma \ref{GGH} of [9], we know that   $Z=\bigcap
_{n<\omega }U^*_n \in {\mathcal Z}(X^*)$. We claim that $\mbox
{int}_{c\sigma X} (Z\backslash\sigma X)=\emptyset$. So suppose to
the contrary that $\mbox {int}_{c\sigma X} (Z\backslash\sigma X)\neq
\emptyset$, and let $U$ be an open set in $\beta X$ such that
\[ \emptyset\neq U\backslash \sigma X \subseteq \mbox {cl}_{\beta X}
U\backslash \sigma X\subseteq Z\backslash \sigma X.\]
For each $n<\omega$, since $\mbox {cl}_X U_{n+1}\subseteq
U_n$, $\mbox {cl}_X U_{n+1}$ and $X\backslash U_n$ are disjoint
zero-sets of $X$ and thus we have $\mbox {cl}_{\beta X} (\mbox
{cl}_X U_{n+1})\cap \mbox {cl}_{\beta X}( X\backslash
U_n)=\emptyset$. Therefore $\mbox {cl}_{\beta X} (\mbox {cl}_X
U_{n+1})\subseteq \beta X \backslash \mbox {cl}_{\beta X}(
X\backslash U_n)$. On the other hand, since $\beta X \backslash
\mbox {cl}_{\beta X}( X\backslash U_n)\subseteq \mbox {cl}_{\beta X}
U_n$, we have $\beta X \backslash \mbox {cl}_{\beta X}( X\backslash
U_n)\subseteq \mbox {int} _{\beta X}( \mbox {cl}_{\beta X} U_n)$,
and therefore $ \mbox {cl}_{\beta X} U_{n+1} \subseteq \mbox {int}
_{\beta X}( \mbox {cl}_{\beta X}  U_n)$. But since
\[\mbox {cl}_{\beta X} U\backslash\sigma X\subseteq Z\backslash\sigma X
\subseteq \mbox {cl}_{\beta X }U_{n+1} \backslash \sigma X\]
it follows that
\[ \mbox {cl}_{\beta X}  U\backslash \mbox
{int} _{\beta X}( \mbox {cl}_{\beta X} U_n)\subseteq \mbox
{cl}_{\beta X}  U\backslash \mbox {cl}_{\beta X} U_{n+1} \subseteq
\sigma X.\]
Therefore by compactness,  for each $n<\omega$, there
exists a countable set $J_n\subseteq I$ such that
\[\mbox {cl}_{\beta X} U\backslash \mbox
{int} _{\beta X}(\mbox {cl}_{\beta X} U_n)\subseteq \mbox
{cl}_{\beta X}\Big(\bigcup_{i \in J_{n}} X_i\Big) \subseteq \mbox
{cl}_{\beta X}\Big(\bigcup_{i \in J} X_i\Big)\]
where $J=J_1\cup J_2\cup\cdots$. Now comparing with above,
for each $n<\omega $, we have
\[ \mbox {cl}_{\beta X} U\subseteq\mbox {cl}_{\beta X }U_n\cup\mbox {cl}_{\beta X}\Big(
\bigcup_{i \in J} X_i\Big)\]
and thus
\[\mbox {cl}_{\beta X} U\subseteq\Big(\bigcap_{n<\omega} \mbox {cl}_{\beta X }U_n\Big)\cup
\mbox {cl}_{\beta X}\Big(\bigcup_{i \in J} X_i\Big).\]
Therefore
\[U\cap X\subseteq \mbox {cl}_{\beta X} U \cap X\subseteq \bigcup_{i \in J}
X_i\cup\Big(\bigcap_{n<\omega}U_n\Big)=\bigcup_{i \in J} X_i\cup \{x_i:
i\in L\}.\]
Let
\[M=\bigcup _{i\in L}\big(\{x_i\}\cup\{x^i_n :n<\omega\}\big).\]
Then as $\{x_i:i\in L\}\subseteq M$ we have
\[\mbox {cl}_{\beta X} U=\mbox {cl}_{\beta X} (U\cap X)\subseteq \mbox {cl}_{\beta X}\Big(
\bigcup_{i \in J} X_i\Big)\cup \mbox {cl}_{\beta X}\big(\{x_i:i\in
L\}\big)\subseteq\sigma X\cup \mbox {cl}_{\beta X}M \]
and so $\mbox {cl}_{\beta X} U\backslash \sigma X\subseteq
\mbox {cl}_{\beta X}M\backslash\sigma X$. For each $n<\omega$, let
$V_n= U_n\cap M$. Then since for each $n<\omega$, $\{x_i:i\in
L\}\subseteq V_n$, it follows that $U\cap X\subseteq V_n\cup
\bigcup_{i \in J} X_i$, and therefore
\[ U\cap X\subseteq \mbox{cl}_{\beta X} U\subseteq
\mbox {cl}_{\beta X}V_n\cup \mbox {cl}_{\beta X}\Big(\bigcup_{i \in J}
X_i\Big)\subseteq \mbox {cl}_{\beta X}V_n\cup\sigma X.\]
Thus by the way we have chosen $U$, for each $n<\omega$,
we have
\[\emptyset\neq U\backslash\sigma X\subseteq \mbox {cl}_{\beta X}
U\backslash\sigma X\subseteq \mbox {cl}_{\beta X}
V_n\backslash\sigma X\subseteq \mbox {cl}_{\beta X}
M\backslash\sigma X.\]
Now since $M$ is a closed subset of the (normal) space
$X$, by Corollary \ref{RPG}.8 of [7], $\mbox {cl}_{\beta X} M$ is a
compactification of $M$ equivalent to $ \beta M$. But $M$ is
zero-dimensional and hence strongly zero-dimensional, as it is a
locally compact metrizable space, (see Theorem 6.\ref{LKG}0 of [7])
therefore, there exists a non-empty clopen subset $W$ of $\mbox
{cl}_{\beta X} M$ such that $W\subseteq U\cap \mbox {cl}_{\beta X}
M$ and $W\backslash\sigma X\neq \emptyset$. Since
\[W\backslash\sigma X\subseteq U\backslash\sigma X\subseteq
\mbox {cl}_{\beta X} V_n\backslash\sigma X\]
and by the way we defined $V_n$ ($V_n$ is a clopen subset
of $M$) $\mbox {cl}_{\beta X} V_n$ is a clopen subset of $\mbox
{cl}_{\beta X} M$, $W\backslash \mbox {cl}_{\beta X} V_n$ is a
compact subset of $\sigma X$. Therefore, for each $n<\omega$, there
exists a countable $H_n\subseteq I$ such that
\[W\backslash \mbox {cl}_{\beta X} V_n\subseteq \mbox {cl}_{\beta X}\Big(\bigcup_{i
\in H_n} X_i\Big)\subseteq\mbox {cl}_{\beta X}\Big(\bigcup_{i \in H}
X_i\Big)\]
where $H=H_1\cup H_2\cup\cdots$. We have
\[ W\subseteq (W\backslash \mbox{cl}_{\beta X} V_n)\cup \mbox {cl}_{\beta X}
V_n\subseteq \mbox {cl}_{\beta X}\Big(\bigcup_{i \in H} X_i\Big)\cup \mbox
{cl}_{\beta X} V_n\]
and therefore
\[W\subseteq \mbox {cl}_{\beta X}\Big(\bigcup_{i \in H} X_i\Big)\cup
\Big(\bigcap_{n<\omega}\mbox {cl}_{\beta X} V_n\Big).\]
Now
\[W\cap M\subseteq \Big(\Big(\bigcup_{i \in H} X_i\Big)\cap M\Big)\cup
\Big(\bigcap_{n<\omega} V_n\Big).\]
We also have
\[\bigcap_{n<\omega} V_n=\bigcap_{n<\omega} U_n\cap M=\{x_i:i\in
L\}\]
and therefore  $W\cap M\subseteq P\cup\{x_i:i\in L\}$,
where $P$ is a countable subset of $M$. For each $i\in L$ for which
$x_i\in W\cap M$, since $W$ is an open subset of $\mbox {cl}_{\beta
X}M$, $W\cap M$ is an open neighborhood of $x_i$ in $M$, and
therefore as $\{x^i_n\}_{n<\omega}$ converges to $x_i$, there exists
an $n_i<\omega$ such that $x^i_{n_i}\in W\cap M$, which implies
$x^i_{n_i}\in P$. Now since $P$ is countable, the set $Q=\{i\in L:
x_{i}\in W\cap M\}$ is also countable and we have
\[W\cap M\subseteq P\cup\{x_i: i\in Q\}\subseteq\bigcup_{i \in G}
X_i\]
for some countable subset $G$ of $I$. But $W$ is chosen to
be clopen in $\mbox {cl}_{\beta X}M$, therefore
\[W=\mbox {cl}_{\beta X}W\cap \mbox {cl}_{\beta X}M=\mbox {cl}_{\beta X}(W\cap M)\subseteq
\mbox {cl}_{\beta X}\Big(\bigcup_{i \in G} X_i\Big)\subseteq\sigma X\]
which is a contradiction, since  $W$ is chosen such that
$W\backslash\sigma X\neq\emptyset$. This shows that $\mbox
{int}_{c\sigma X} (Z\backslash\sigma X)=\emptyset$. Now we note that
by Theorem  6.7 of [9], $Z\in \lambda ({\mathcal E}(X))$ implies
that
\[\mbox {cl}_{\beta X}\big(\{x_i:i\in L\}\big)=\mbox {cl}_{\beta X}
\Big(\bigcap_{n<\omega}U_n\Big)\subseteq\sigma X \]
and so
\[\mbox {cl}_{\beta X}\big(\{x_i:i\in L\}\big)\subseteq\mbox {cl}_{\beta X}\Big(\bigcup_{i
\in F} X_i\Big)\]
for some countable $F\subseteq I$. Therefore since
$\bigcup_{i \in F} X_i$ is clopen in $X$,
\[\{x_i: i\in L\}\subseteq\bigcup_{i \in F} X_i\]
which is clearly  a contradiction, since we are assuming
that  $L$ is uncountable. This shows that   $Z\notin
\lambda({\mathcal E}(X))$, which completes the
proof.
\end{proof}

In the next theorem we give a characterization of the image of
$\lambda$. Note that if $X$ is locally compact, then $X^*$ is closed
in $\beta X$ and thus it is $C$-embedded.

\begin{theorem}\label{LFKG}
Let $X$ be a locally compact
non-compact  metrizable space. Then  $\lambda({\mathcal E}(X))$
consists of exactly those non-empty zero-sets of $\beta X$ which
miss $X$.
\end{theorem}

\begin{proof}
Suppose that $ S \in \lambda ({\mathcal E}(X))$. Then
$S \in {\mathcal Z}(X^*)$, and therefore  there exists  an $f\in
C(\beta X,{\bf I})$ such that $Z(f)\backslash X=S$. Let for each
$n<\omega $, $U_n= X\cap f^{-1}([0,1/n))$. Then as in the proof of
Lemma  \ref{FFFGH} of [9], $\{U_n\}_{n<\omega}$ is a regular sequence of
open sets in $X$ such that $S=\bigcap _{n<\omega} U^*_n$. Now since
$Z(f)\cap X=\bigcap _{n<\omega} U_n$, it follows from  Theorem  6.7
of [9] that $Z(f)\cap X$ is $\sigma$-compact. Let $Z(f)\cap
X=\bigcup _{n<\omega} K_n$, where each $K_n$ is compact, and let for
each $n<\omega$, $g_n\in C(\beta X,{\bf I})$  be such that
$g_n(K_n)\subseteq\{1\}$ and $g_n(Z(f)\backslash X)\subseteq\{0\}$.
Let $g=\sum g_n/2^n$. Then $g$ is continuous and $S=Z(f)\cap Z(g)$
is a zero-set in $\beta X$ which misses $X$.

Conversely, suppose that  $\emptyset\neq S \in {\mathcal Z}(\beta
X)$ is such that $S\cap X=\emptyset$. Let $S=Z(f)$, for some
 $f\in C(\beta X, {\bf I})$. For each  $n<\omega$, we let
$U_n=X\cap f^{-1}([0,1/n))$. Then by the proof of Lemma \ref{FFFGH} of
[9], $\{U_n\}_{n<\omega}$ is a regular sequence of open sets in
$X$ such that $S=\bigcap _{n<\omega} U^*_n$. Now since  $\bigcap
_{n<\omega} U_n=\emptyset$,  $\{U_n\}_{n<\omega}$ is an extension
trace in $X$, and thus $S \in \lambda ({\mathcal
E}(X))$.
\end{proof}

\section{On the order structure of the sets ${\mathcal E}_K(X)$ and ${\mathcal E}(X)$ and their relationship}

In  Theorems \ref{TUUG} and \ref{WG} of [9], for a locally compact separable
non-compact metrizable space $X$, the authors characterized the
image of ${\mathcal E}_K (X)$  under $\lambda$ as the set of all
non-empty clopen subsets of $X^*$. We will extend this result to the
non-separable case in bellow. First in the following lemma we characterize the elements of ${\mathcal E}_K (X)$  in terms of the extension traces generating them.

\begin{lemma}\label{LYFG}
Let $X$ be a locally compact
non-compact metrizable space and let  $Y=X\cup\{p\}\in {\mathcal E}
(X)$. Then the following conditions are equivalent.
\begin{itemize}
\item[\rm(1)] $Y$ is locally compact;
\item[\rm(2)] For every extension trace ${\mathcal U}= \{U_n\}_{n<\omega}$ in $X$ generating $Y$, there exists a $k<\omega$ such that for all $n\geq k$, $\mbox {{\em cl}}_X U_n\backslash U_{n+1}$ is compact;
\item[\rm(3)] There exists an extension trace ${\mathcal U}=\{U_n\}_{n<\omega}$ in $X$ generating $Y$, such that for all $n<\omega$, $\mbox {{\em cl}}_X U_n\backslash U_{n+1}$ is compact.
\end{itemize}
\end{lemma}

\begin{proof}
\emph{(1) implies (2).} Let ${\mathcal U}=
\{U_n\}_{n<\omega}$ be an extension trace in $X$ which generates
$Y$. Since $\{U_n\cup\{p\}\}_{n<\omega}$ forms an open  base at $p$
in $Y$, there exists a $k<\omega$ such that $\mbox {cl}_Y(
U_k\cup\{p\})$ is compact. Now for each $n\geq k$, $\mbox {cl}_X
U_n\backslash U_{n+1}$ is a closed subset of $\mbox {cl}_Y
(U_k\cup\{p\})$, and therefore is compact.

That (2) implies (3) is trivial. {\em (3) implies (1).} Let
${\mathcal U}= \{U_n\}_{n<\omega}$ be an extension trace in $X$
which generates $Y$, and suppose that $\mbox {cl}_X U_n\backslash
U_{n+1}$ is compact for all $n<\omega$. Let $W=U_1\cup\{p\}$, and
suppose that  $\{V_i\}_{i\in I}$ is an open cover of  $\mbox {cl}_Y
W$ in $Y$. Let $j\in I$ be such that $p\in V_j$, and let $m<\omega$
be such that $U_m\cup\{p\}\subseteq V_j$. Now since each of $\mbox
{cl}_X U_n\backslash U_{n+1}$ is compact, a finite subset of
$\{V_i\}_{i\in I}$  covers $\mbox {cl}_Y
W$.
\end{proof}

\begin{theorem}\label{LDKG}
Let $X$ be a locally compact
non-compact metrizable space. Then  $\lambda({\mathcal E}_K (X))$
consists of exactly those elements of $\lambda({\mathcal E} (X))$
 which are clopen in  $X^*$.
\end{theorem}

\begin{proof}
We assume that $X$ is non-separable. Assume the
notations of Theorem \ref{IUUG}. Suppose that ${\mathcal U}=
\{U_n\}_{n<\omega}$ is an extension trace in $X$ which generates
$Y\in {\mathcal E}_K (X)$. By  Lemma  \ref{LYFG},  we may assume that
$\mbox {cl}_X U_n\backslash U_{n+1}$  is compact for all $n<\omega$.
Since $\bigcap _{n<\omega} U_n=\emptyset$, we have $\mbox {cl}_X
U_1=\bigcup _{n<\omega } (\mbox {cl}_X U_n\backslash U_{n+1})$. But
for each $n<\omega$, $\{ X_i\}_{i\in I}$ is an open cover of $\mbox
{cl}_X U_n\backslash U_{n+1}$, and therefore there exist finite
subsets $J_n\subseteq I$ such that $\mbox {cl}_X U_n\backslash
U_{n+1}\subseteq \bigcup _{i\in J_n} X_i$. Let $J=J_1\cup J_2\cup
\cdots$, and let $M=\bigcup _{i\in J} X_i$. Then clearly ${\mathcal U}$
is also an extension trace in $M$, for which by the above lemma, the
corresponding one-point metrizable extension of $M$ is locally compact.
 Now since $M$ is separable, by  Theorem \ref{TUUG}
of [9], $P=\bigcap_{n<\omega}\mbox {cl}_{\beta M}U_n\backslash M$
is a clopen subset of  $\mbox {cl}_{\beta X} M\backslash M$, which
is itself a clopen subset of $X^*$ as $M$ is clopen in $X$, and
therefore it is a clopen subset of $X^*$. We note that $\lambda
(Y)=P$.

Now suppose that $Z\in \lambda({\mathcal E} (X))$ is clopen in
$X^*$. First we note that by Lemma  6.6 of [9], we have
$Z\backslash \sigma X=\mbox {int}_{c\sigma X} (Z\backslash\sigma
X)=\emptyset$, and therefore $Z\subseteq\sigma X$. It follows from
the latter that there exists a countable $J\subseteq I$ such that
$Z\subseteq \mbox {cl}_{\beta X}(\bigcup_{i \in J} X_i)$. Let
$M=\bigcup_{i \in J} X_i$. Now $Z$ is a clopen subset of $\mbox
{cl}_{\beta X} M\backslash M$, and since $M$ is separable, it
follows from Theorem \ref{WG} of [9] and Lemma  \ref{LYFG} that $Z=\bigcap
_{n<\omega} (\mbox {cl}_{\beta M} U_n\backslash M)$, for some
extension trace ${{\mathcal U}}= \{U_n\}_{n<\omega}$ in $M$ for
which $\mbox {cl}_M U_n\backslash U_{n+1}$ is compact for each
$n<\omega$. But ${\mathcal U}$ is an extension trace in $X$, and
since $\mbox {cl}_X U_n\backslash U_{n+1}= \mbox {cl}_M
U_n\backslash U_{n+1}$ is compact, its corresponding one-point
metrizable extension of $X$, denoted by  $Y$, is locally compact. Now we
note that $Z=\lambda (Y)$.
\end{proof}

The following lemma is implicit in the proof of Theorem \ref{LDKG}.

\begin{lemma}\label{FFG}
Let  $X$ be a locally compact
non-separable metrizable space. Assume the notations of Theorem \ref{IUUG}.
Then for each countable $J\subseteq I$, we have $ (\bigcup_{i \in J}
X_i)^*\in \lambda({\mathcal E}(X))$.
\end{lemma}

\begin{lemma}\label{FFYG}
Suppose that $X$ is a locally
compact non-compact metrizable space and let $Z\in \lambda({\mathcal
E}(X))$. If $S\in {\mathcal Z}(X^*)$ is such that $\emptyset\neq
S\subseteq Z$, then  $S\in \lambda({\mathcal E}(X))$.
\end{lemma}

\begin{proof}
Let $T\in{\mathcal Z}(\beta X)$ be such that
$S=T\backslash X$. Now $S=Z\cap T$ misses $X$, and thus  by Theorem \ref{LFKG} $S\in
\lambda({\mathcal E}(X))$.
\end{proof}

The following theorem gives another characterization of the image of ${\mathcal E}_K (X)$ under $\lambda$.

\begin{theorem}\label{LPPG}
Let  $X$ be a locally compact
non-compact metrizable space. Then $\lambda ({\mathcal E}_K (X))$
consists of exactly those non-empty zero-sets of $X^*$ which are of
the form $X^*\backslash \mbox {\em cl}_{\beta X} (Z(f))$, where
$f\in C(X, \mathbf{I})$ is of $\sigma$-compact support.
\end{theorem}

\begin{proof}
Suppose that $S\in\lambda({\mathcal E}_K (X))$. Then
since $\{\mbox {cl}_{\beta X} Z\backslash X:Z\in {\mathcal Z}(X)\}$
forms a base for closed subsets of $X^*$, there exists a collection
$\mathcal{C}$ of zero-sets of $X$ such that $S=\bigcup \{
X^*\backslash \mbox {cl}_{\beta X} Z:Z\in \mathcal{C}\}$. Since $S$
is compact, there exists a finite number of zero-sets $Z_1,\ldots,Z_n$
such that $S=\bigcup_{i=1}^n ( X^*\backslash \mbox {cl}_{\beta X}
Z_i)=X^*\backslash \mbox {cl}_{\beta X} Z$, where $Z=Z_1\cap\cdots\cap
Z_n\in {\mathcal Z}(X)$. Let $Z=Z(f)$, for some $f\in C(X,
\mathbf{I})$. If $X$ is separable, then trivially $\mbox {supp}(f)$
is $\sigma$-compact. So suppose that $X$ is non-separable and assume
the notations of Theorem \ref{IUUG}. Let $L=\{i\in I: Coz (f)\cap X_i\neq
\emptyset\}$. Then there exists a zero-set $T\in {\mathcal Z}(X)$
such that $T\subseteq Coz(f)$ and for each $i\in L$, $T\cap
X_i\neq\emptyset$. Now since $T\cap Z(f)=\emptyset$, we have $\mbox
{cl}_{\beta X}T\cap \mbox {cl}_{\beta X}Z(f)=\emptyset$, and
therefore $\mbox {cl}_{\beta X}T\subseteq \beta X\backslash\ \mbox
{cl}_{\beta X}Z(f)$. But $S$ is a clopen subset of $X^*$, and so by
Lemma 6.6 of [9], $S\subseteq\sigma X$. This implies that $\mbox
{cl}_{\beta X}T\subseteq \mbox {cl}_{\beta X}(\bigcup _{i\in
J}X_i)$, for some countable $J\subseteq I$. It follows that
$T\subseteq \bigcup _{i\in J}X_i$, and therefore $L$ is countable.
Now clearly $\mbox {supp}(f)$, being a closed subset of the
separable space $\bigcup _{i\in J}X_i$, is $\sigma$-compact (see
\ref{FPG}.C of [7]).

Conversely,
suppose that $\emptyset\neq S\in {\mathcal Z}(X^*)$ is of the form
$X^*\backslash \mbox {cl}_{\beta X} (Z(f))$, for some   $f\in
C(X,\mathbf{I})$ of $\sigma$-compact support. If $X$ is separable
then clearly $S$, being a clopen subset of $X^*$, is in $\lambda
({\mathcal E}_K (X))$. Suppose that  $X$ is non-separable. Then by
definition of $\sigma X$, since $\mbox {supp}(f)$ is
$\sigma$-compact, we have $ S\subseteq \mbox {cl}_{\beta X} (\mbox
{supp}(f))\subseteq\sigma X$, and therefore $S\subseteq \mbox
{cl}_{\beta X}(\bigcup _{i\in J}X_i)$, for some countable
$J\subseteq I$.
 Thus  $S\in\lambda ({\mathcal E}_K
(X))$.
\end{proof}

In the following theorem, assuming [CH], we give a  purely order-theoretic
description of ${\mathcal E}_K(X)$ as a subset of ${\mathcal E}(X)$.

\begin{theorem}\label{RPG} {\em [CH]}
Let  $X$ be a locally
compact non-separable metrizable space. For a set ${\mathcal
F}\subseteq {\mathcal E}(X)$ consider the following conditions.
\begin{itemize}
\item[\rm(1)] For each $A\in\mathcal{F}$, $|\{ Y\in {\mathcal E}(X): Y>A\}|\leq \aleph _1;$
\item[\rm(2)] If $A\in {\mathcal E}(X)$ is such that $|\{ Y\in {\mathcal E}(X): Y>A\}|\leq\aleph _1$, then there exists a $B\in {\mathcal F}$ such that $B<A$;
\item[\rm(3)] For each $A, B\in\mathcal{F}$ such that $A< B$, there exists a $C\in\mathcal{F}$ such that $B\wedge C=A$ and $B$ and $C$ have no common upper bound in ${\mathcal E}(X)$.
\end{itemize}
Then the set ${\mathcal E}_K(X)$ is the largest (with
respect to set-theoretic  inclusion) subset of ${\mathcal E}(X)$
satisfying the above three conditions.
\end{theorem}

\begin{proof}
First
we verify that ${\mathcal E}_K(X)$ satisfies the above conditions.
To show  that condition (1) is satisfied, let $A \in {\mathcal
E}_K(X)$. Then we have $\lambda (A)\subseteq\sigma X$ (see Lemma 6.6
of [9]) and therefore, assuming the notations of Theorem \ref{IUUG},
$\lambda (A)\subseteq \mbox {cl}_{\beta X} M$, where $M=\bigcup
_{i\in G} X_i$, for some countable $G\subseteq I$.  Now if $Y\in
{\mathcal E}(X)$ is such that $Y>A$, then $\lambda (Y)$ is a
zero-set in $\mbox {cl}_{\beta X} M$. But $| {\mathcal Z}(\mbox
{cl}_{\beta X} M)|\leq \aleph _1$, as $M$ is separable, and thus
condition  (1) holds.

Now we
show that ${\mathcal E}_K(X)$ satisfies condition (2). So suppose
that $A\in {\mathcal E}(X)$ is such that $|\{ Y\in {\mathcal E}(X):
Y>A\}|\leq \aleph _1$. First we show that $\lambda
(A)\subseteq\sigma X$. Suppose the contrary, and let ${\mathcal
V}=\{V_n\}_{n<\omega}$ be an extension trace in $X$ which generates
$A$. For each $n<\omega$, let $H_n= \{i\in I : V_n \cap
X_i\neq\emptyset\}$.  Then since we are assuming that  $\lambda (A)
\backslash \sigma X \neq \emptyset$, each $H_n$ is an  uncountable
subset of $I$. We consider the following two cases.

{\em Case 1)} Suppose that  $\bigcap_{n<\omega}H_n$ is uncountable.
Let $K\subseteq\bigcap_{n<\omega}H_n$, with $|K|=\aleph _1$. For
each non-empty $L\subseteq K$ and each $n<\omega$ let
\[ W^n_L=\Big(\bigcup_{i\in L}X_i\Big)\cap V_n.\]
Then each ${\mathcal W}_L=\{W^n_L\}_{n<\omega}$ is an
extension trace in $X$ finer than $\mathcal{V}$, and ${\mathcal
W}_{L_1}$ and ${\mathcal W}_{L_2}$ are  non-equivalent for distinct
non-empty $L_1, L_2\subseteq K$. But this is a contradiction, as by
our assumption the number of these extension traces cannot be
greater than $\aleph _1$.

{\em Case 2)} Suppose that  $\bigcap_{n<\omega}H_n$ is countable. We
define a sequence  $\{k_n\}_{n<\omega}$ of positive integers as
follows. Let $k_1=1$. Then since $H_1\supseteq H_2\supseteq\cdots$ and
$H_{k_1}$ is uncountable, arguing inductively there exists a
sequence $k_1<k_2<\cdots$  with $H_{k_n}\backslash H_{k_{n+1}}$ being
uncountable for each $n<\omega$. We may assume that  $H_n\backslash
H_{n+1}$ is uncountable for each $n<\omega$. Suppose that $|K|=
\aleph _1$, and let for each $n<\omega$, $K_n\subseteq H_n\backslash
H_{n+1}$ be such  $|K_n|= \aleph _1$. We use $K$ as an index set to
(faithfully) index the elements of $K_n$.
 Thus  $K_n=\{k_i^n:i\in K\}$. For each non-empty $L\subseteq K$ and each $n<\omega$ let
$L_n=\{k_i^n:i\in L\}$, and define
\[ W^n_L=\bigcup\{V_n\cap X_i : i\in L_n\cup L_{n+1}\cup\cdots\}.\]
Then each ${\mathcal W}_L=\{W^n_L\}_{n<\omega}$ is an
extension trace in $X$ finer than $\mathcal{V}$, and they are
non-equivalent for distinct non-empty $L_1, L_2\subseteq K$. But
this is again a contradiction.

Therefore  $\lambda (A)\subseteq\sigma X$ and we can assume that
$\lambda (A)\subseteq P^*$ properly, where $P=\bigcup_{i\in H}X_i$,
and $H\subseteq I$ is countable. Let $\lambda (B)= P^*$. Then $B\in
{\mathcal E}_K(X)$ and $B<A$. Thus ${\mathcal E}_K(X)$ satisfies
condition  (2).

Next, to show that ${\mathcal E}_K(X)$ satisfies condition  (3),
suppose that $A, B\in {\mathcal E}_K(X)$ are such that $A<B$.  Let
$C\in {\mathcal E}_K(X)$ be such that $\lambda (C)=\lambda
(A)\backslash\lambda (B)$. Then clearly $B\wedge C=A$ and thus
condition (3) holds for ${\mathcal E}_K(X)$.

Now suppose that a set ${\mathcal F}\subseteq {\mathcal E}(X)$
satisfies conditions (1)-(3). Let $A\in\mathcal{F}$. Then by
condition (1), $|\{ Y\in {\mathcal E}(X): Y>A\}|\leq \aleph _1$.
Arguing as above we have  $\lambda (A)\subseteq\sigma X$. Let
$\lambda (A)\subseteq Q^*$, where $Q=\bigcup_{i\in J}X_i$ and
$J\subseteq I$ is countable. Let $B\in {\mathcal E}(X)$ be such that
$\lambda (B)= Q^*$. Then since $|\{ Y\in{\mathcal E}(X): Y>B\}|\leq
\aleph _1$, using condition (2), there exists a $C\in\mathcal{F}$
such that $C<B$. Therefore $C< A$, and so by condition (3), there
exists a $D\in\mathcal{F}$ such that $A\wedge D=C$ and $A$ and $D$
have no common upper bound in ${\mathcal E}(X)$. Therefore $\lambda
(A)\cap\lambda (D)=\emptyset$. Suppose that
$x\in\lambda(B)\backslash (\lambda(A)\cup\lambda(D))$. Let $f\in
C(\beta X,\textbf{I})$ be such that $f(x)=1$ and
$f(\lambda(A)\cup\lambda(D))=\{0\}$. Let $S=Z(f)\cap\lambda (C)$.
Then since $C\leq A$, $S\neq\emptyset$, and therefore $S=\lambda
(E)$, for some $E\in{\mathcal E}(X)$.  Clearly since
 $\lambda(A)\subseteq S$, we have  $E\leq
A$. But $\lambda(D)\subseteq Z(f)$ and $C\leq D$, therefore
$\lambda(D)\subseteq S$, and thus  $E\leq D$. This combined with
$E\leq A$ implies  that $E\leq C$. But $x\in
\lambda(B)\subseteq\lambda(C)$ and $x\notin S$. This contradiction
shows that $\lambda(B)\backslash\lambda(A)\subseteq\lambda(D)$.
Finally, we note  that by the above inclusion
$\lambda(B)\backslash\lambda(D)\subseteq\lambda(A)$, and conversely,
if $x\in\lambda(A)$, then since $B\leq A$, and
$\lambda(A)\cap\lambda(D)=\emptyset$, we have
$x\in\lambda(B)\backslash\lambda(D)$. Therefore
$\lambda(A)=\lambda(B)\backslash\lambda(D)$, and thus $\lambda(A)$
is clopen in $X^*$. This shows that $A\in{\mathcal E}_K(X)$, and
therefore ${\mathcal F}\subseteq {\mathcal E}_K(X)$, which together
with the first part of the proof, establishes the
theorem.
\end{proof}

\begin{theorem}\label{TTRPG}
Let  $X$ be a locally compact
non-compact metrizable space. Then $ {\mathcal E}_K(X)$ and $
{\mathcal E}(X) $ are never order-isomorphic.
\end{theorem}

\begin{proof}
{\em Case 1)} Suppose that $X$ is separable. Suppose
that ${\mathcal E}_K(X)$ is order-isomorphic  to ${\mathcal E}(X)$,
and let $\psi : \lambda({\mathcal
E}_K(X))\rightarrow\lambda({\mathcal E}(X))$ denote an
order-isomorphism. First we show that $ \lambda({\mathcal
E}_K(X))=\lambda({\mathcal E}(X))$, from which it follows that every
non-empty zero-set of $X^*$ is clopen in $X^*$, and therefore $X^*$
is a $P$-space. By Proposition 1.65 of [16] every pseudocompact
$P$-space is finite, thus $X^*$ is finite. By 4C of [16], the
Stone-\v{C}ech remainder of a non-pseudocompact space has at least
$2^{2^{\aleph_0}}$ points. Therefore $X$ is pseudocompact and being
metrizable it is compact. But this is a contradiction. Now let
$X^*\neq Z\in \lambda({\mathcal E}(X))$. Let $A=X^*\backslash
\psi^{-1}(Z)$, where $A\in \lambda({\mathcal E}_K(X))$. If
$\psi(A)\cap Z\neq\emptyset$, then there exists a $B\in
\lambda({\mathcal E}_K(X))$ such that $\psi(B)=\psi(A)\cap Z$. But
such a $B$ necessarily has non-empty intersection with one of
$\psi^{-1}(Z)$ or $A$. Now since $\psi$ is an order-isomorphism, it
is easy to see that in either case we get a contradiction. Therefore
$\psi(A)\cap Z=\emptyset$. If $\psi(A)\cup Z\neq X^*$, then  there
exists an $\emptyset\neq H\in {\mathcal Z}(X^*)$ with $H\cap
(\psi(A)\cup Z)=\emptyset$. Let $G\in \lambda({\mathcal E}_K(X))$ be
such that $\psi(G)=H$, then again we get a contradiction, as $G$
intersects one  of $\psi^{-1}(Z)$ or $A$. Therefore $\psi(A)\cup Z=
X^*$, and thus $Z=X^*\backslash\psi(A)$, i.e.,
$Z\in\lambda({\mathcal E}_K(X))$.

{\em Case 2)} Suppose that $X$ is non-separable. Suppose to the
contrary that $ {\mathcal E}_K(X) $ and ${\mathcal E}(X) $ are
order-isomorphic and let  $\phi: {\mathcal E}_K(X) \rightarrow
{\mathcal E}(X) $ denote an order-isomorphism. Since $X$ is
non-separable, there exists a sequence $\{Y_n \}_{n<\omega}$ in $
{\mathcal E}(X)$ such that $Y_1<Y_2<\cdots$. Consider ${\mathcal
F}=\{\lambda(Y_n )\}_{n<\omega}$. Then ${\mathcal F}$ has the
f.i.p., and therefore $Z=\bigcap{\mathcal F}\in \lambda({\mathcal
E}(X) )$. Let $Z=\lambda(Y)$, for some $Y\in {\mathcal E}(X)$.
Clearly $Y=\bigvee_{n<\omega}Y_n$. Let for each $n<\omega$, $\phi
(S_n)=Y_n$ and $\phi (S)=Y$. Then $S_1<S_2<\cdots<S$. For each
$n<\omega$, let $\lambda(S_n )\backslash\lambda(S)=\lambda(T_n )$,
for some $T_n\in {\mathcal E}_K(X)$. Now since the sequence
$\{S_n\}_{n<\omega}$ is increasing, the sequence $\{\lambda(T_n
)\}_{n<\omega}$  and equivalently the sequence
$\{\phi(T_n)\}_{n<\omega}$ is  also increasing, and thus
$\bigvee_{n<\omega}\phi(T_n)\in {\mathcal E}(X)$. Let $T\in
{\mathcal E}_K(X)$ be such that
$\phi(T)=\bigvee_{n<\omega}\phi(T_n)$. Let  $A\in{\mathcal E}_K(X)$
be  such that $\lambda(A)=\lambda(S)\cup\lambda(T)$. Now for each
$n<\omega$,  $T\geq T_n$, and  thus $
\lambda(A)\subseteq\lambda(S)\cup\lambda(T_n)=\lambda(S_n)$, i.e.,
for each $n<\omega$,  we have $A\geq S_n$, or equivalently,
$\phi(A)\geq \phi(S_n)=Y_n$. Therefore $\phi(A)\geq Y=  \phi(S)$ and
$A\geq S$. Thus $T\geq S$. But $T\geq T_1$, which is a contradiction
as $\lambda(T_1)\cap
\lambda(S)=\emptyset$.
\end{proof}

\begin{lemma}\label{FPG}
Let $X$ be  a locally compact
non-separable  metrizable space. If $\emptyset\neq Z\in {\mathcal
Z}(\beta X)$ then $Z\cap \sigma X\neq \emptyset$.
\end{lemma}

\begin{proof}
Suppose that $\{x_n\}_{n<\omega}$ is an infinite
sequence in $\sigma X$. Then using the notations of Theorem \ref{IUUG},
there exists a countable $J\subseteq I$ such that
$\{x_n\}_{n<\omega}\subseteq \mbox {cl}_{\beta X}(\bigcup_{i\in J}
X_i)$, and therefore it has a limit point in $\sigma X$. Thus
$\sigma X$ is countably compact and therefore, pseudocompact, and
$\upsilon (\sigma X)=\beta (\sigma X)=\beta X$. The result now
follows as for any Tychonoff space $T$, any non-empty zero-set of
$\upsilon T$ intersects $T$ (see Lemma \ref{EOYG}(f) of
[14]).
\end{proof}

\begin{lemma}\label{OEFPG}
Let  $X$ be  a locally compact
non-separable  metrizable space. If $\emptyset\neq Z\in {\mathcal
Z}(X^*)$ then $Z\cap \sigma X\neq \emptyset$.
\end{lemma}

\begin{proof}
Let $S\in Z(\beta X)$ be such that $Z=S\backslash X$.
By the above lemma $S\cap \sigma X\neq\emptyset$. Suppose that
$S\cap (\sigma X\backslash X)=\emptyset$. Then $S\cap \sigma X=S\cap
X$. Assume the notations of Theorem \ref{IUUG} and let $L=\{i\in I :S\cap
X_i\neq\emptyset\}$. Since $S\cap (\sigma X\backslash X)=\emptyset$,
it follows that $L$ is finite. We define a function $f:\beta
X\rightarrow{\bf I}$ such that it equals to 1 on $\mbox {cl}_{\beta
X}(\bigcup_{i\in L} X_i) $, and it is 0 otherwise. Clearly $f$ is
continuous. Since $Z(f)\cap S\in {\mathcal Z}(\beta X)$ misses
$\sigma X$, by the above lemma, $Z(f)\cap S=\emptyset$. But since
$\beta X\backslash \sigma X\subseteq Z(f)$, we have $Z=S\cap(\beta
X\backslash \sigma X)\subseteq S\cap Z(f)=\emptyset$, which is a
contradiction. Therefore $Z\cap (\sigma X\backslash X)=S\cap (\sigma
X\backslash X)\neq\emptyset$.
\end{proof}

\begin{lemma}\label{FGLG}
Let $X$ be  a locally compact
non-separable  metrizable space and let  $ S, T\in {\mathcal
Z}(X^*)$. If $S\cap \sigma X\subseteq T\cap \sigma X$ then
$S\subseteq T$.
\end{lemma}

\begin{proof}
Suppose that $S\backslash T\neq\emptyset$. Let $x\in
S\backslash T$ and let $f\in C(\beta X,{\bf I})$  be such that
$f(x)=0$ and $f(T)=\{1\}$. Then $Z(f)\cap S\in {\mathcal Z}(X^*)$ is
non-empty and therefore by Lemma \ref{OEFPG}, $Z(f)\cap S\cap \sigma
X\neq\emptyset$. But this is impossible as  we have $Z(f)\cap S\cap
\sigma X\subseteq Z(f)\cap
T=\emptyset$.
\end{proof}

\begin{theorem}\label{IPG}
Let  $X$ be a non-compact
metrizable space. Then $ {\mathcal E}(X)$ has a  minimum if and only
if $X$ is locally compact and separable.
\end{theorem}

\begin{proof}
Suppose that $Y=X\cup\{p\}$ is the minimum in $
{\mathcal E}(X) $. If $X$ is not locally compact, then there exists
an $x\in X$  such that for every open neighborhood $U$ of $x$ in
$X$, $\mbox {cl}_X U$ is not compact. Let $U$ and $W$ be disjoint
open neighborhoods of $x$ and $p$ in $Y$, respectively. Since $\mbox
{cl}_X U$ is not compact, there exists a discrete sequence
$\{V_n\}_{n<\omega}$ of non-empty open (in  $\mbox {cl}_X U$)
subsets of $X$, which are  faithfully indexed. Consider $ {\mathcal
F}=\{V_n\cap U\}_{n<\omega}$. Then ${\mathcal F}$ is a discrete
sequence of non-empty open  subsets of $X$. For each $n<\omega$, let
$A_n$ be a non-empty open subset of $X$ such that $\mbox {cl}_X A_n
\subseteq V_n\cap U$. For each $n<\omega$ we form a sequence
$\{B_k^n \}_{k<\omega}$ of  non-empty open subset of $X$ such that
$A_n \supseteq B_1^n $, and $B_k^n\supseteq \mbox {cl}_X B_{k+1}^n$
for each $k<\omega$. Let $m<\omega$ be such that $B( p,
1/m)\subseteq W$, and for each $n<\omega$ define
\[C_n=B\Big(p, \frac{1}{m+n}\Big)\cap X \mbox { and }  E_n = \bigcup
\{B_n^k:k\geq n\}.\]
Let for each $n<\omega$, $D_n=C_n\cup E_n$. Then since
$\{B_n^k \}_{k<\omega}$ is discrete, we have  $\mbox {cl}_X
D_{n+1}\subseteq D_n$, and $\bigcap_{n<\omega}D_n=\emptyset$, i.e.,
${\mathcal D}=\{D_n\}_{n<\omega}$ forms an extension trace in $X$.
Since ${\mathcal C}=\{C_n\}_{n<\omega}$ is an extension trace in $X$
generating $Y$, by Theorem \ref{LPPG} of [9], ${\mathcal D}$ is finer
than ${\mathcal C}$, and therefore there exists a $k<\omega$ such
that $D_k\subseteq C_1$. But this is a contradiction, as
$C_1\subseteq W$ and $D_k\cap U\neq\emptyset$. Therefore $X$ is
locally compact.

Suppose that $X$ is not separable. Since $ {\mathcal E}(X) $ has a
minimum, $\lambda( {\mathcal E}(X))$ has a maximum. Let $S$ denote
the maximum of $\lambda({\mathcal E}(X))$. Assume the notations of
Theorem \ref{IUUG}. Now since for each countable $J\subseteq I$,
$(\bigcup_{i\in J}X_i)^*\subseteq S$, we have $\sigma X\backslash
X\subseteq S$, and therefore by Lemma \ref{FGLG} (with $S$ and $X^*$ being
the zero-sets) we have $S=X^*$, which is a contradiction, as $X$ is
not $\sigma$-compact (see 1B of [16]). Thus $X$ is locally compact
and separable.

The converse is clear, as in this case, the one-point
compactification of $X$ is the
minimum.
\end{proof}

By replacing ${\mathcal E}(X)$ by ${\mathcal E}_K(X)$ in the last
part of the above proof we  obtain the following result.

\begin{theorem}\label{IPOG}
Let $X$ be a locally compact
non-separable metrizable space. Then ${\mathcal E}_K(X)$ has no
minimum.
\end{theorem}

In the next result we show that when $X$ is a zero-dimensional locally compact metrizable space, ${\mathcal E}_K(X)$ is a cofinal subset of ${\mathcal E}(X)$.  For this   purpose  we need the following proposition, stated in
Lemma 1\ref{OYG} of [6].

\begin{proposition}\label{FPOG}
Let  $X$ be a locally compact space, let $F$ be a nowhere  dense subset of $X$, and let $Z$ be a non-empty zero-set  of $\beta X$ which misses $X$.  Then we have
\[\mbox{\em int}_{X^*}(Z\backslash\mbox{\em cl}_{\beta X}F)\neq
 \emptyset.\]
\end{proposition}

\begin{theorem}\label{IPPOG}
Let  $X$ be a zero-dimensional
locally compact non-separable metrizable space. Then for each $Y\in
{\mathcal E}(X)$, there exists an $S\in {\mathcal E}_K(X)$ such that
$S\geq Y$. In other words, ${\mathcal E}_K(X)$ is a cofinal subset
of ${\mathcal E}(X)$. Furthermore, there is no such greatest $S$,
(in fact there are at least $2^{\aleph_0}$
 mutually incomparable elements of ${\mathcal E}_K(X)$
greater than $Y$) and if $Y$ is not locally compact, there is no
such least $S$.
\end{theorem}

\begin{proof}
Suppose that $Y\in {\mathcal E}(X)$, and let
$Z=\lambda(Y)$. By Proposition \ref{FPOG}  we have $\mbox
{int}_{X^*}Z\neq\emptyset$. Now since $X$ is strongly
zero-dimensional, (see Theorem 6.\ref{LKG}0 of [7]) there exists a clopen
subset $V$ of $ \beta X$ such that $\emptyset\neq V\backslash
X\subseteq Z$. Let $S\in {\mathcal E}_K(X)$ be such that $\lambda
(S)=V\cap Z=V\backslash X$. Then $S\geq Y$. Now since $V\cap X$ is
non-compact, there exists a discrete family $\{U_n\}_{n<\omega}$ of
non-empty open subsets of $V\cap X$. Since $X$ is locally compact
and zero-dimensional,  we may assume that each  $U_n$ is  compact.
Let $\omega=\bigcup_{t<2^{\aleph_0}} N_t$ be a partition of $\omega$
into almost disjoint  infinite sets. Let for $ t<2^{\aleph_0}$,
$A_t=\bigcup_{n\in N_t} U_n$. Then each $A_t$ is a  clopen subsets
of $ X$ and $\mbox {cl}_{\beta X} A_{s}\cap \mbox {cl}_{\beta X}
A_{t}\subseteq X$, for $s\neq t$. Let $S_t\in {\mathcal E}_K(X)$ be
such that $\lambda(S_t)=\mbox {cl}_{\beta X} A_t\cap
Z=A_t^*\subseteq V\backslash X$. Clearly $S_t>S$, for each
$t<2^{\aleph_0}$, and they are mutually incomparable for $s\neq t$.

Now suppose that  $Y$ is not locally compact. Then $Z$ is not clopen
in $X^*$ and therefore $V\backslash X\neq Z$. By Proposition \ref{FPOG} we
have  $\mbox {int}_{X^*}(Z\backslash V)\neq\emptyset$. Let $U$ be a
clopen subset of  $\beta X$ such that $\emptyset\neq U\backslash
X\subseteq Z\backslash V$. Then $(U\cup V)\backslash X\subseteq Z$
is clopen in $X^*$ and properly contains $\lambda
(S)$.
\end{proof}

\begin{theorem}\label{IPIG}
Let  $X$ be a zero-dimensional
locally compact non-compact  metrizable space and let  $S,T\in
{\mathcal E}(X)$. Then $S\geq T$ if and only if for every $Y\in
{\mathcal E}_K(X)$, if $Y\geq S$ then $Y\geq T$.
\end{theorem}

\begin{proof}
One implication  is trivial. Suppose that  for every
$Y\in {\mathcal E}_K(X)$,
 $Y\geq S$ implies $Y\geq T$. If $\lambda (S)\backslash \lambda
 (T)\neq \emptyset$, then there exists an $A\in {\mathcal Z}(\beta X)$ such that $A\cap\lambda
 (S)\neq \emptyset$ and $A\cap\lambda
 (T)=\emptyset$. Now  by Proposition \ref{FPOG}, we have  $\mbox
{int}_{X^*}(A\cap\lambda
 (S))\neq\emptyset$, and thus  there exists a clopen subset $V$ of $\beta X$ such that $\emptyset\neq V\backslash X\subseteq A\cap\lambda
 (S)$. Let $Y\in {\mathcal E}_K(X)$ be such that $ \lambda (Y)=V\backslash X$. Then  $Y\geq S$,  and therefore by our assumptions  $Y\geq
 T$. But $\lambda (Y)\cap\lambda (T)=\emptyset$, which is a
 contradiction.  Thus $\lambda (S)\backslash \lambda
 (T)=\emptyset$  and $S\geq T$.
\end{proof}

\begin{corollary}\label{OIPIG}
Let $X$ be a zero-dimensional
locally compact non-compact  metrizable space. Then for any
$S\in{\mathcal E}(X)$ we have
\[ S=\bigwedge\big\{Y\in{\mathcal E}_K(X):Y\geq S\big\}.\]
\end{corollary}

In the next two theorems we investigate the question of existence of
greatest lower bounds and least upper bounds for arbitrary subsets
of ${\mathcal E}_K(X)$ and ${\mathcal E}(X)$.

\begin{lemma}\label{FGRLG}
Let  $X$ be a zero-dimensional
locally compact non-compact   metrizable space and let
$\emptyset\neq Z\in {\mathcal Z}(X^*)$. Then the following
conditions are equivalent.
\begin{itemize}
\item[\rm(1)] $Z\in \lambda({\mathcal E}(X))$;
\item[\rm(2)] There exists an extension trace $\{V_n\}_{n<\omega}$ in $X$, consisting of clopen subsets of $X$, such that $Z=\bigcap_{n<\omega}V_n^*$.
\end{itemize}
\end{lemma}

\begin{proof}
That (2) implies (1) is trivial. {\em (1) implies
(2).} Let
 $\{U_n\}_{n<\omega}$  be an extension trace in $X$ such that
$Z=\bigcap_{n<\omega}U_n^*$. Since $X$ is strongly zero-dimensional,
(see Theorem 6.\ref{LKG}0 of [7]) and for each $n<\omega$, $\mbox
{cl}_XU_{n+1}$ and $X\backslash U_n$ are  completely separated in
$X$, by Lemma 6.\ref{LFKG} of [7],
 there exists a clopen subset $V_n$ of $X$ such that
$\mbox {cl}_XU_{n+1}\subseteq V_n\subseteq U_n$. Clearly now
$\{V_n\}_{n<\omega}$ forms an extension trace in $X$ which satisfies
our requirements.
\end{proof}

A Boolean algebra is said to be {\em Cantor separable} if no
strictly increasing sequence has a least upper bound (see \ref{LFKG}0 of
[16]). Proposition \ref{LFKG}2 of [16] states that the Boolean algebra of
clopen subsets of a totally disconnected compact space without
isolated points, in which every zero-set is regular-closed, is
Cantor separable. This will be used in the following theorem.

\begin{theorem}\label{IPY}
Let $X$ be a  locally compact
non-compact  metrizable space. Then the following hold.
\begin{itemize}
\item[\rm(1)] For $Y_1, Y_2\in {\mathcal E}_K(X)$,  $Y_1, Y_2$ may not have any common upper bound in ${\mathcal E}_K(X)$;
\item[\rm(2)] For any $Y_1,\ldots, Y_n\in {\mathcal E}_K(X)$ which have a common upper bound in ${\mathcal E}_K(X)$, $\bigvee_{i=1}^n Y_i$ exists in ${\mathcal E}_K(X)$;
\item[\rm(3)] For a sequence $\{Y_n\}_{n<\omega}$ in ${\mathcal E}_K(X)$ which
has an upper bound in ${\mathcal E}_K(X)$, $\bigvee_{n<\omega} Y_n$
may not  exists in ${\mathcal E}_K(X)$. In fact, if we assume $X$ to
be moreover zero-dimensional, then for any sequence
$\{Y_n\}_{n<\omega}$ in ${\mathcal E}_K(X)$ with $Y_1<Y_2<\cdots$,
$\{Y_n\}_{n<\omega}$ has an upper bound in ${\mathcal E}_K(X)$ but
$\bigvee_{n<\omega} Y_n$ does not  exists in ${\mathcal E}_K(X)$;
\item[\rm(4)] For any $Y_1,\ldots,Y_n\in
{\mathcal E}_K(X)$, $\bigwedge_{i=1}^n Y_i$ exists in ${\mathcal
E}_K(X)$;
\item[\rm(5)] For any sequence
$\{Y_n\}_{n<\omega}$ in ${\mathcal E}_K(X)$, $\{Y_n\}_{n<\omega}$
has a lower bound in ${\mathcal E}_K(X)$;
\item[\rm(6)] For a  sequence $\{Y_n\}_{n<\omega}$ in ${\mathcal E}_K(X)$,
$\bigwedge_{n<\omega} Y_n$ may not  exists in ${\mathcal E}_K(X)$.
In fact, if  we assume $X$  to be moreover  zero-dimensional, then
for any sequence $\{Y_n\}_{n<\omega}$ in ${\mathcal E}_K(X)$ such
that $Y_1>Y_2>\cdots$, $\bigwedge_{n<\omega} Y_n$ does not  exist in
${\mathcal E}_K(X)$;
\item[\rm(7)] An uncountable family  of elements of  ${\mathcal E}_K(X)$ may
not have any   common lower bound in ${\mathcal E}_K(X)$. In fact,
if we assume $X$ to be moreover non-separable and  zero-dimensional,
then there exists a subset of ${\mathcal E}_K(X)$ of cardinality
$\aleph_1$,  with no common lower bound in ${\mathcal E}_K(X)$.
\end{itemize}
\end{theorem}

\begin{proof}
(1), (2) and (4) are straightforward. 3) Suppose that
$X$ is zero-dimensional and  let $\{Y_n\}_{n<\omega}$ be a sequence
in ${\mathcal E}_K(X)$ such that $Y_1<Y_2<\cdots$. Since  the sequence
$\{\lambda (Y_n)\}_{n<\omega}$ is decreasing, it has the f.i.p., and
therefore $S=\bigcap _{n<\omega}\lambda (Y_n)\in \lambda({\mathcal
E}(X))$. Now  Proposition \ref{FPOG} implies that $\mbox
{int}_{X^*}S\neq\emptyset$. But since $X$ is strongly
zero-dimensional (see Theorem 6.\ref{LKG}0 of [7]) there exists a
non-empty clopen subset $U$ of $X^*$ such that $U\subseteq S$. Let
$A\in {\mathcal E}_K (X)$ be such that $\lambda (A)=U$. Then clearly
$A$ is an upper bound for $\{Y_n\}_{n<\omega}$ in ${\mathcal
E}_K(X)$. Now suppose that  $\bigvee_{n<\omega} Y_n$ exists in
${\mathcal E}_K(X)$ and let $Y=\bigvee_{n<\omega} Y_n$. Consider the
family $\{\lambda (Y_n)\backslash \lambda (Y)\}_{n<\omega}$ of
non-empty decreasing clopen subsets of $X^*$. Let  $T=\bigcap
_{n<\omega}\lambda (Y_n)\backslash \lambda (Y)\neq\emptyset$. Then
by Proposition \ref{FPOG}, we have $\mbox {int}_{X^*}T\neq\emptyset$. Let
$V$ be a non-empty clopen subset of $X^*$ such that $V\subseteq T$.
Let $\lambda (B)=V\cup\lambda (Y)$, for some $B\in{\mathcal
E}_K(X)$. Then since for any $n<\omega$, $V\subseteq
T\subseteq\lambda (Y_n)$, $B$ is an upper bound for
$\{Y_n\}_{n<\omega}$, and therefore $B\geq Y$. Thus
 $V\subseteq \lambda (B)\subseteq\lambda (Y)$. But this is a
contradiction as $V\subseteq T \subseteq X^*\backslash\lambda (Y)$.

5) We assume  that $X$ is non-separable. Let $\{Y_n\}_{n<\omega}$ be
a sequence in ${\mathcal E}_K(X)$. Then for each $n<\omega$, by
Lemma 6.6 of [9], we have $\lambda (Y_n)\subseteq\sigma
X\backslash X$. Assuming the notations of Theorem \ref{IUUG}, it  follows
that there exists a countable $J\subseteq I$ such that for each
$n<\omega$, $\lambda(Y_n)\subseteq M^*$, where $M=\bigcup_{i\in
J}X_i$.  Let $\lambda(Y)= M^*$, for some $Y\in {\mathcal E}_K(X)$.
Then clearly $Y$ is a lower bound for the sequence
$\{Y_n\}_{n<\omega}$.

6) Suppose that  $X$ is  zero-dimensional and let
$\{Y_n\}_{n<\omega}$ be a sequence in ${\mathcal E}_K(X)$ with
$Y_1>Y_2>\cdots$. Suppose that $Y=\bigwedge_{n<\omega} Y_n$   exists in
${\mathcal E}_K(X)$. First we assume that  $X$ is separable and
verify that $X^*$ is a totally disconnected compact space without
isolated points in which every zero-set is regular-closed.

Clearly
$X^*$ is totally disconnected, as it is zero-dimensional (see
Theorem  6.\ref{LKG}0 of [7]). Since $X$ is Lindel\"{o}f, by \ref{FPG}.C of [7],
it is $\sigma$-compact. By  Remark 1\ref{GGH}7 of [6], the
Stone-\v{C}ech remainder of any zero-dimensional locally compact
$\sigma$-compact space has no isolated points. Therefore  $X^*$ does
not have any isolated points. Finally, $X$ being Lindel\"{o}f is
realcompact. By Theorem 1\ref{TQYG} of [6], any zero-set of the
Stone-\v{C}ech remainder of a locally compact realcompact space is
regular-closed, therefore, every zero-set in $X^*$ is
regular-closed. Now since $\lambda (Y_1)\subseteq\lambda
(Y_2)\subseteq\cdots\subseteq\lambda (Y)$ (properly), Proposition  \ref{LFKG}2
of [16] implies the existence of an $S\in{\mathcal B}(X^*)$ such that
$\lambda (Y_n)\subseteq S\subseteq\lambda (Y)$ (properly), for each
$n<\omega$. Let $A\in {\mathcal E}_K(X)$ be such that $\lambda
(A)=S$. Then clearly $A$ is a lower bound for the sequence
$\{Y_n\}_{n<\omega}$ but $A>Y$. This contradiction proves our
theorem in this case.

Now suppose that  $X$ is non-separable. By Lemma  6.6 of
[9], we have $\lambda(Y)\subseteq\sigma X\backslash X$ and
$\lambda(Y_n)\subseteq\sigma X\backslash X$, for any $n<\omega$.
Assume the notations of Theorem \ref{IUUG}. Then  $\lambda(Y)\subseteq M^*$
and $\lambda(Y_n)\subseteq M^*$, for any $n<\omega$, where
$M=\bigcup_{i\in J}X_i$ and $J\subseteq I$ is countable. Then since
$\mbox {cl}_{\beta X}M\simeq \beta M$ and $M$ is separable, the
problem reduces to the case we considered above.

7) Let $X$ be  zero-dimensional. Assume the notations of Theorem
\ref{IUUG}. Let  $J\subseteq I$ be such that $|J|=\aleph_1$, and let
$\{J_k: k<\omega_1\}$ be a partition of $J$ into mutually disjoint
subsets with $|J_k|=\aleph_0$, for any $k<\omega_1$. For any
$k<\omega_1$, let $Y_k\in {\mathcal E}_K(X)$ be such that $\lambda
(Y_k)=(\bigcup_{i\in J_k}X_i)^*$. We claim that the family
 ${\mathcal F}= \{Y_k\}_{k<\omega_1}$ has no lower bound in
${\mathcal E}_K(X)$. Suppose the  contrary, and let $Y\in {\mathcal
E}_K(X)$ be a lower bound for  ${\mathcal F}$. By Lemma  \ref{OIPIG}, there
exists an extension trace ${\mathcal U}= \{U_n\}_{n<\omega}$ in $X$
generating  $Y$, such that each $U_n$ is a clopen subsets of $X$. By
Lemma \ref{LYFG}, there exists an $m<\omega$ such that $U_n\backslash
U_{n+1}$ is compact, for all $n\geq m$. We may assume that  $m=1$.
Let $k<\omega_1$. Then
\[\Big(\bigcup_{i\in J_k}X_i\Big)^*=\lambda (Y_k)\subseteq\lambda
(Y)=\bigcap_{n<\omega}U_n^*\subseteq U_1^*\]
and therefore
\[\mbox {cl}_{\beta X}\Big(\bigcup_{i\in J_k}X_i\Big)\subseteq \mbox {cl}_{\beta X}U_1 \cup
X=\mbox {cl}_{\beta X}U_1 \cup \bigcup_{i\in I}X_i.\]
Now since $U_1$ is clopen in $X$, $\mbox {cl}_{\beta
X}U_1$ is clopen in $\beta X$, and therefore there exists a finite
set $L_k\subseteq I$ such that
\[\mbox {cl}_{\beta X}\Big(\bigcup_{i\in
J_k}X_i\Big)\subseteq \mbox {cl}_{\beta X}U_1 \cup \bigcup_{i\in
L_k}X_i\subseteq \mbox {cl}_{\beta X}\Big(U_1 \cup \bigcup_{i\in
L_k}X_i\Big).\]
But $U_1 \cup \bigcup_{i\in L_k}X_i$ is clopen in $X $,
and thus
\[\bigcup_{i\in J_k}X_i\subseteq U_1 \cup\bigcup_{i\in L_k}X_i.\]
Let for any $k<\omega_1$, $i_k\in J_k\backslash L_k$. Then
$X_{i_k}\subseteq U_1$, and therefore
$\bigcup\{X_{i_k}:k<\omega_1\}$ being a closed subset of the
$\sigma$-compact set $U_1=\bigcup_{n<\omega}(U_n\backslash U_{n+1})$
is Lindel\"{o}f. But this is clearly a contradiction. This completes
the proof.
\end{proof}

The following is a counterpart of the above theorem, which deals with the subsets of ${\mathcal E}(X)$.

\begin{theorem}\label{YIOPY}
Let $X$ be a  locally compact
non-compact  metrizable space. Then the following hold.
\begin{itemize}
\item[\rm(1)] For $Y_1, Y_2\in {\mathcal E}(X)$,  $Y_1, Y_2$ may not have any
common upper bound in ${\mathcal E}(X)$;
\item[\rm(2)] For any sequence $\{Y_n\}_{n<\omega}$ in ${\mathcal E}(X)$, if
$\{Y_n\}_{n<\omega}$ has an  upper bound in ${\mathcal E}(X)$, then
$\bigvee _{n<\omega}Y_n$ exists in  ${\mathcal E}(X)$;
\item[\rm(3)] For any $Y_1,\ldots, Y_n\in {\mathcal E}(X)$, $\bigwedge_{i=1}^n
Y_i$ exists in ${\mathcal E}(X)$;
\item[\rm(4)] A  sequence $\{Y_n\}_{n<\omega}$ in ${\mathcal E}(X)$ may not
have any lower bound in  ${\mathcal E}(X)$. In fact, if we moreover assume that  $w(X)\geq 2^{\aleph_0}$, then there exists a sequence
$\{Y_n\}_{n<\omega}$ in ${\mathcal E}(X)$ which does not have any
lower bound in ${\mathcal E}(X)$.
\item[\rm(5)] A  sequence $\{Y_n\}_{n<\omega}$ in ${\mathcal E}(X)$ which has a
lower bound in ${\mathcal E}(X)$, may not have a greatest lower
bound in ${\mathcal E}(X)$.
\end{itemize}
\end{theorem}

\begin{proof}
(1)-(3) are straightforward. 4)  Let $\Delta$ denote
the set of all increasing (i.e., $f(n)\leq f(n+1)$, for any
$n<\omega$) functions $f:\omega\rightarrow\omega$ which are not
eventually constant. We first check that $|\Delta|=2^{\aleph_0}$. To
show this, let for each $g\in \{0,1\}^\omega$ which is not
eventually constant, $f_g:\omega\rightarrow\omega$ be defined by $
f_g(n)=n+g(n)$, for any $n<\omega$. Then clearly since for distinct
$g,h\in\{0,1\}^\omega$, $f_g\neq f_h$, we have $|\Delta|\geq
2^{\aleph_0}$. It is clear that $|\Delta|\leq2^{\aleph_0}$. Assume
the notations of Theorem \ref{IUUG}. Since $w(X)=|I|\geq 2^{\aleph_0}$, for
simplicity we may assume that $I\supseteq\Delta$. For any
$n,k<\omega$, let
\[ U^n_k=\bigcup\big\{X_f:f\in \Delta \mbox { and } f(k)\leq n\big\}\]
and let ${\mathcal U}_n= \{U_k^n\}_{k<\omega}$. We verify
that ${\mathcal U}_n$ is an extension trace in $X$. By the way we
defined $U_k^n$ and since $f$ is increasing we have
$U_{k+1}^n\subseteq U_k^n$. Suppose that
$\bigcap_{k<\omega}U_k^n\neq\emptyset$, and let
$x\in\bigcap_{k<\omega}U_k^n$. Since for any ${k<\omega}$, $x\in
 U_k^n$, by definition of $U_k^n$, there exists an $f_k\in \Delta$
 such that $f_k(k)\leq n$ and $x\in X_{f_k}$. But since  the family $\{X_i\}_{i\in I}$ is  faithfully indexed
 and $X_i\cap X_j=\emptyset$, for distinct $i,j\in I$,  we obtain that
 $f_1=f_2=\ldots\equiv h$. Now for any ${k<\omega}$, we have  $h(k)=f_k(k)\leq
 n$,
 which implies $h$ to be eventually constant, which is a contradiction. Therefore $\bigcap_{k<\omega}U_k^n=\emptyset$
  and each ${\mathcal U}_n$ is an  extension traces in $X$.  Let for any
  ${n<\omega}$, $Y_n\in{\mathcal E}(X)$ be generated by ${\mathcal U}_n$. We claim that
  $\{Y_n\}_{n<\omega}$ has no lower bound in ${\mathcal E}(X)$. So
  suppose to the contrary that $Y\in {\mathcal E}(X)$ is such that $Y\leq Y_n$, for any
  $n<\omega$. Let ${\mathcal U}= \{U_n\}_{n<\omega}$ be an extension
  trace in $X$ which generates $Y$. Since for any $n<\omega$,
  we are assuming $Y\leq Y_n$, by   Theorem \ref{LPPG} of [9], ${\mathcal U}_n $ is finer than ${\mathcal
  U}$. For any $n,i<\omega$, let
  $k^n_i<\omega$ be such that $U^i_{k^n_i}\subseteq U_n$. We can
  also assume that $k^n_1<k^n_2<k^n_3<\cdots$, for any $n<\omega$. We
  define a function
$g:\omega\rightarrow\omega$ as follows.

Let $g(i)=1$ for $i=1,\ldots, k^1_{t_1}$, where $t_1=1$. Inductively
assume that for $n<\omega$, $t_1<\cdots<t_n$ are defined in such a way
that
\[ g(i)=m, \mbox { for } i=k^{m-1}_{t_{m-1}}+1,\ldots, k^m_{t_m}
\mbox { and  } m=1,\ldots,n.\]
Now since
$k^{n+1}_1<k^{n+1}_2<k^{n+1}_3<\cdots$, there exists a $t<\omega$ such
that $k_t^{n+1}>k^n_{t_n}$ and $t>t_n$. Let $t_{n+1}=t$ and define
\[g(i)=n+1, \mbox { for } i=k^n_{t_n}+1,\ldots, k^{n+1}_{t_{n+1}}.\]
Consider the function $g:\omega\rightarrow\omega$. Clearly $g\in
\Delta$ and since for any $n<\omega$, $g(k_{t_n}^n)=n$ we have
$X_g\subseteq U^n_{k_{t_n}^n}$. But since the sequence
$\{t_n\}_{n<\omega}$ is increasing, $t_n\geq n$, and therefore
$k_{t_n}^n\geq k_n^n$. Thus $U^n_{k_{t_n}^n}\subseteq U^n_{k_n^n}$,
which combined with the fact that $U^n_{k_n^n}\subseteq U_n$ implies
that $X_g\subseteq U_n$, for any $n<\omega$. But this is a
contradiction, as $\bigcap_{n<\omega} U_n=\emptyset$.

5) Let $X$   be non-separable and let $\{Y_n\}_{n<\omega}$ be a
sequence  in ${\mathcal E}_K(X)$ such that $Y_1>Y_2>\cdots$. By part
(5) of Theorem \ref{IPY}, $\{Y_n\}_{n<\omega}$  has a lower bound in
${\mathcal E}(X)$. Suppose that $A=\bigwedge_{n<\omega} Y_n$ exists
in ${\mathcal E}(X)$.
 Then since for each
$n<\omega$, $Y_n>A$, we have $\lambda(A)\backslash\lambda
(Y_n)\neq\emptyset$, and therefore
$S=\bigcap_{n<\omega}(\lambda(A)\backslash\lambda (Y_n))\in
\lambda({\mathcal E}(X))$. By Proposition \ref{FPOG}, there exists a
non-empty open subset $U$ of $X^*$ with $U\subseteq S$. Let $x\in U$
and let $f\in C(X^*, \textbf{I})$ be such that $f(x)=1$ and
$f(X^*\backslash U)\subseteq\{0\}$. Let $T= \lambda(A)\cap Z(f)$.
Clearly for each $n<\omega$, $\lambda(Y_n)\subseteq X^*\backslash
U\subseteq Z(f)$, and therefore $\lambda(Y_n)\subseteq T$. Let $Y\in
{\mathcal E}(X)$ be such that $\lambda(Y)=T$.  Then since for each
$n<\omega$, $Y\leq Y_n$, we have $Y\leq A$. But since $x\in U
\subseteq\lambda(A)$, this implies that  $x\in \lambda(Y)\subseteq
Z(f)$, which is a contradiction.
\end{proof}

Our final result of this section deals with the cardinalities of cofinal subsets of ${\mathcal E}(X)$.

\begin{theorem}\label{FGY}
Let  $X$ be a locally compact
non-separable  metrizable space and let $ {\mathcal F}\subseteq
{\mathcal E}(X)$. If for each $Y\in {\mathcal E}(X)$ there exists an
$A\in {\mathcal F}$ such that $A\leq Y$, ($A\geq Y$, respectively)
then $ {\mathcal F}$ is uncountable.
\end{theorem}

\begin{proof}
Suppose that for each $Y\in {\mathcal E}(X)$ there
exists an $A\in {\mathcal F}$ such that $A\leq Y$. Suppose that $
{\mathcal F}=\{Y_n\}_{n<\omega}$. Let for each $n<\omega$,
$S_n=\lambda(Y_n)$. Then by Lemma 6.6 of [9],  $\mbox
{int}_{c\sigma X}(S_n\backslash \sigma X)=\emptyset$. Now since
$c\sigma X$ is compact, by the Baire  Category Theorem we have
$\mbox {int}_{c\sigma X}(\bigcup_{n<\omega}S_n\backslash \sigma
X)=\emptyset$, and therefore $c\sigma
X\backslash\bigcup_{n<\omega}S_n\neq\emptyset$. Let $x\in c\sigma
X\backslash\bigcup_{n<\omega}S_n$. Then since for each $n<\omega$,
$x\notin S_n$, there exists a $Z_n\in{\mathcal Z}(\beta X)$ such
that $ x\in Z_n$ and $Z_n \cap S_n =\emptyset$. Let
$Z=\bigcap_{n<\omega}Z_n$. Then since $Z\backslash  X\in {\mathcal
Z}(X^*)$ is non-empty, by Lemma  \ref{OEFPG} we have $Z\cap(\sigma
X\backslash X)\neq\emptyset$. Therefore, using the notations of
Theorem \ref{IUUG},  for some countable $J\subseteq I$,
$T=Z\cap(\bigcup_{i\in J}X_i)^*\neq \emptyset$. Let $Y\in {\mathcal
E}(X)$ be such that $\lambda(Y)=T$. By assumption, $Y\geq Y_k$, for
some $k<\omega$. Therefore $S_k=\lambda(Y_k)\supseteq \lambda(Y)=T$.
But $T\cap S_k=\emptyset$, which is a contradiction.

To show the second part of the theorem, let for each $i\in I$,
$Y_i\in {\mathcal E}(X)$ be such that $\lambda(Y_i)=X_i^*$. Let
$A_i\in {\mathcal F}$ be such that $ A_i\geq Y_i$. Then clearly
since for $i\neq j$, $X_i^*\cap X_j^*=\emptyset$, we have
$|{\mathcal F}|\geq |\{A_i:i\in
I\}|=|I|=w(X)$.
\end{proof}

\section{The relationship between the order structure of the set  ${\mathcal E}_K(X)$ and the topology of subspaces of $\beta X\backslash X$}

In Theorem \ref{EWG} of [9] the authors proved that for locally compact
separable  metrizable spaces $X$ and $Y$ whose Stone-\v{C}ech
remainders are zero-dimensional,  ${\mathcal E}_K(X)$ and ${\mathcal
E}_K(Y)$ are order-isomorphic if and only if $X^*$ and $Y^*$ are
homeomorphic. In this section we generalize this result to the case
when $X$ and $Y$ are not separable.

\begin{theorem}\label{GGH}
Let  $X$ and $Y$ be
zero-dimensional locally compact non-separable  metrizable spaces
and let $\omega\sigma X=\sigma X \cup\{\Omega\}$ and $\omega\sigma
Y=\sigma Y \cup\{\Omega '\}$ be the one-point compactifications of
$\sigma X$ and $\sigma Y$, respectively. If ${\mathcal E}_K(X)$ and
${\mathcal E}_K(Y)$ are order-isomorphic then $\omega\sigma
X\backslash X$ and $\omega\sigma Y\backslash Y$ are homeomorphic.
\end{theorem}

\begin{proof}
Let $\phi:{\mathcal E}_K(X)\rightarrow{\mathcal
E}_K(Y)$ be  an order-isomorphism  and let $g=\lambda_Y \phi
\lambda_X^{-1}:\lambda_X({\mathcal
E}_K(X))\rightarrow\lambda_Y({\mathcal E}_K(Y))$.  We define a
function $ G:{\mathcal B}(\omega\sigma X\backslash
X)\rightarrow{\mathcal B}(\omega\sigma Y\backslash Y)$ between the
two Boolean algebras of clopen sets, and verify that it is an
order-isomorphism.

Set $G(\emptyset)=\emptyset$ and $G(\omega\sigma X\backslash
X)=\omega\sigma Y\backslash Y$. Let $U\in {\mathcal B}(\omega\sigma
X\backslash X)$. If $U\neq\emptyset$ and $\Omega\notin U$, then $U$
is an open subset of $\sigma X\backslash X$, and therefore an open
subset of $X^*$. Assuming  the notations of Theorem \ref{IUUG}, there
exists a countable $J\subseteq I$ such that $U\subseteq
(\bigcup_{i\in J}X_i)^*$, and thus $U\in \lambda_X ({\mathcal
E}_K(X))$. In this case we let $ G(U)= g(U)$. If $U\neq\omega\sigma
X\backslash X$ and $\Omega\in U$, then $(\omega\sigma X\backslash
X)\backslash U\in \lambda_X({\mathcal E}_K(X))$ and we let $
G(U)=(\omega\sigma Y\backslash Y)\backslash g((\omega\sigma
X\backslash X)\backslash U)$.

To show that $G$ is an
order-homomorphism,  let $U,V\in{\mathcal B}(\omega\sigma
X\backslash X)$ with $U\subseteq V$. We may assume that
$U\neq\emptyset$ and $V\neq\omega\sigma X\backslash X$. We consider
the following three cases.

{\em Case 1)} If $\Omega\notin V$, then clearly $G(U)=g(U)\subseteq
g(V)=G(V)$.

{\em Case 2)} Suppose that  $\Omega\notin U$ and $\Omega\in V$. If
$G(U)\backslash G(V)\neq\emptyset$ then $T=g(U)\cap g((\omega\sigma
X\backslash X)\backslash V)\neq\emptyset$ and therefore
$T\in\lambda_Y({\mathcal E}_K(Y))$. Let $S\in \lambda_X({\mathcal
E}_K(X))$ be such that $g(S)=T$. Then since $g$ is an
order-isomorphism, we have  $S\subseteq U\cap((\omega\sigma
X\backslash X)\backslash V)=\emptyset$, which is a contradiction.
Therefore $G(U)\subseteq G(V)$.

{\em Case 3)} If $\Omega\in U$, then since  $(\omega\sigma
X\backslash X)\backslash V\subseteq(\omega\sigma X\backslash
X)\backslash U$ we have
\[G(U)=(\omega\sigma Y\backslash Y)\backslash g\big((\omega\sigma
X\backslash X)\backslash U\big)\subseteq(\omega\sigma Y\backslash
Y)\backslash g\big((\omega\sigma X\backslash X)\backslash V\big)=G(V).\]
This shows that $G$ is an order-homomorphism.

To complete the proof we note that since $\phi^{-1}:{\mathcal
E}_K(Y)\rightarrow {\mathcal E}_K(X)$ is also an order-isomorphism,
if we denote $h=\lambda_X \phi^{-1}\lambda_Y^{-1}$, then arguing as
above, $h$ induces an order-homomorphism $H:{\mathcal
B}(\omega\sigma Y\backslash Y)\rightarrow{\mathcal B}(\omega\sigma
X\backslash X)$ which is easy to see that $H=G^{-1}$.

To see that $\omega\sigma X\backslash X$ is zero-dimensional, we note that since $X$ is zero-dimensional locally compact metrizable, it is strongly zero-dimensional (see Theorem 6.\ref{LKG}0 of [7]) i.e., $\beta X$ is zero-dimensional. Thus $\sigma X$ is also zero-dimensional. But the one-point compactification of a locally compact non-compact zero-dimensional space is again zero-dimensional, therefore $\omega\sigma X$ and thus $\omega\sigma X\backslash X$ is zero-dimensional. Similarly $\omega\sigma Y\backslash Y $ is also
zero-dimensional, and thus,  they  are homeomorphic by Stone
Duality.
\end{proof}

The following provides a  converse to the above theorem under some weight
restrictions. Note that here we are not assuming $X$ and $Y$ to be necessarily zero-dimensional.

\begin{theorem}\label{FFFGH}
Let $X$ and $Y$ be locally compact
non-separable  metrizable spaces. Suppose moreover, that at least
one of $X$ and $Y$ has weight greater than $2^{\aleph_0}$. Then if
$\omega\sigma X\backslash X$ and  $\omega\sigma Y\backslash Y$ are
homeomorphic, ${\mathcal E}_K(X)$ and ${\mathcal E}_K(Y)$ are
order-isomorphic.
\end{theorem}

\begin{proof}
Without any loss of generality we may  assume that
$w(X)>2^{\aleph_0}$. Suppose that $f:(\sigma X\backslash
X)\cup\{\Omega\}\rightarrow(\sigma Y\backslash Y)\cup\{\Omega '\}$
is a homeomorphism. First we show that $f(\Omega)=\Omega '$.

Suppose that $f(\Omega)=p$, where $p\in \sigma Y\backslash Y$.
Suppose that $K$ is a countable subset of $J$ such that $p\in \mbox
{cl}_{\beta Y} (\bigcup_{i\in K} Y_i)$, where $Y=\bigoplus_{i\in J}
Y_i$, with each $Y_i$ being a separable non-compact subspace. Then
$V= (\bigcup_{i\in K} Y_i)^*$ is an open neighborhood of $p$ in
$(\sigma Y\backslash Y)\cup\{\Omega '\}$, and therefore, there
exists a neighborhood $W$ of $\Omega$ in $(\sigma X\backslash
X)\cup\{\Omega \}$ such that $f(W)\subseteq V$. Clearly we may
choose $W$ to be of the form
\[ W=\big((\sigma X\backslash
X)\cup\{\Omega\}\big)\backslash\mbox {cl}_{\beta X}\Big(\bigcup_{i\in L}
X_i\Big)\]
for some countable $L\subseteq I$ (with the notations of
Theorem \ref{IUUG}). Let $M=\bigcup_{i\in K} Y_i$. Then since $M$ is
separable we have
 $w(\mbox {cl}_{\beta Y} M)\leq 2^{\aleph_0}$. Now $\{X^*_i:
i\in I\backslash L\}$ is a collection of non-empty mutually disjoint
open subsets of $W$, and thus  $w(W)\geq |I|=w(X)>2^{\aleph_0}$. On
the other hand $w(W)=w(f(W))\leq w(\mbox {cl}_{\beta
Y}M)\leq2^{\aleph_0}$. This contradiction shows that
$f(\Omega)=\Omega '$. Therefore $\sigma X\backslash X$ is
homeomorphic to $\sigma Y\backslash Y$.

Now suppose that
$Z\in\lambda_X({\mathcal E}_K(X))$. Then by Lemma 6.6 of [9], we
have  $Z\subseteq \sigma X\backslash X$, and thus  there exists a
countable set $A\subseteq I$ such that $Z\subseteq \mbox {cl}_{\beta
X} P$, where $P=\bigcup_{i\in A} X_i$.  But since $P^*$ is clopen in
$\sigma X\backslash X$, $f(P^*)$ is clopen in $\sigma Y\backslash
Y$, and since it is also compact, there exists a countable set
$B\subseteq J$ such that $f(P^*)\subseteq Q^*$, where
$Q=\bigcup_{i\in B} Y_i$.  But since $Z$ is clopen in $X^*$, $f(Z)$
is clopen in $\sigma Y\backslash Y$, and as $f(Z)\subseteq Q^*$, it
is also clopen in $ Q^*$, and thus clopen in $Y^*$, i.e., $f(Z)\in
\lambda_Y ({\mathcal E}_K(Y))$. Now we define a function $F:
\lambda_X ({\mathcal E}_K(X))\rightarrow\lambda_Y ({\mathcal
E}_K(Y))$ by $F(Z)=f(Z)$. The function $F$ is clearly well-defined
and it is an  order-homomorphism. Since $f^{-1}$ is also a
homeomorphism which takes  $\Omega '$ to $\Omega$, arguing as above,
we can define a function $G: \lambda_Y ({\mathcal
E}_K(Y))\rightarrow\lambda_X ({\mathcal E}_K(X))$ by
$G(Z)=f^{-1}(Z)$, which is clearly the inverse of $F$. Thus $F$ is
an order-isomorphism.
\end{proof}

A compact zero-dimensional $F$-space of weight $2^{\aleph_0}$ in
which every non-empty $G_\delta$-set has infinite interior, is
called a {\em Parovi\v{c}enko space}. It is well known that under
 [CH] $\omega^*$ is the only Parovi\v{c}enko space (Parovi\v{c}enko Theorem, see Corollary
1.2.4 of [15]).

The following theorem is proved (assuming [CH]) in [5], for the
case when $X$ is a discrete space of cardinality $\aleph _1$.

\begin{theorem}\label{FGGF}
{\em [CH]} Let $X$ be a zero-dimensional locally compact metrizable space of weight $\aleph
_1$. Then we have
\[(\sigma X\backslash X)\cup \{\Omega\}\simeq \omega ^*.\]
\end{theorem}

\begin{proof}
We verify the assumptions of the Parovi\v{c}enko
Theorem. Let $Y=(\sigma X\backslash X)\cup \{\Omega\}$. To show that
$Y$ is an $F$-space, let $A$ and $B$ be disjoint cozero-sets in $Y$.
Suppose first that $\Omega$ belong to one of $A$ and $B$. Without
any loss of generality we may assume that $\Omega\in A$. Then since
$Y\backslash A\in {\mathcal Z}(Y)$, using the notations of Theorem
\ref{IUUG},  we have $Y\backslash A\subseteq M^*$, where $M=\bigcup_{i\in
J}X_i$, and $J\subseteq I$ is countable. Now  $M$ is
$\sigma$-compact (see \ref{FPG}.C of [7]) and $\beta M\backslash M\simeq
M^*$ is an $F$-space, (see 1.62 of [16]) thus  $A\cap M^*$ and $B$,
being disjoint cozero-sets in $M^*$, are completely separated in
$M^*$. But $M^*$ itself is clopen in $Y$, thus $A$ and $B$ are
completely separated in $Y$. Suppose that $\Omega$ does not belong
to any of $A$ and $B$. Then since $A$ and $B$ are cozero-sets in
$Y$, they are $\sigma$-compact, and therefore as $A,B\subseteq
\sigma X\backslash X$, they are cozero-sets in $P^*$, where
$P=\bigcup_{i\in K}X_i$, for some countable $K\subseteq I$. But as
above $P^*$ is an $F$-space, and $A$, $B$ being completely separated
in $P^*$, are completely separated in $Y$.

Next we show that $Y$ is zero-dimensional of weight $\aleph_1$. For
each countable $L\subseteq I$, let $Q_L=\bigcup_{i\in L}X_i$. Then
since $Q_L$ is separable, $\mbox {cl}_{\beta X}Q_L$ has weight at
most $\aleph_1$. Now since $Q_L$ is strongly zero-dimensional,  (see
 Theorem 6.\ref{LKG}0 of [7]) by Theorem \ref{IUUG}.15 of [7], we can choose a base
${\mathcal C}_L$ consisting of clopen subsets of $Q_L^*$ (and
therefore clopen in $Y$) such that $|{\mathcal C}_L|\leq\aleph _1$.
Let
\[ {\mathcal D}=\bigcup \{{\mathcal C}_L:L\subseteq I \mbox {  is
countable}\}\cup \{Y\backslash \mbox {cl}_{\beta X}Q_L:L\subseteq I
\mbox { is countable}\}.\]
Then clearly ${\mathcal D}$ forms a base consisting of
clopen subsets of $Y$, and therefore  $Y$ is zero-dimensional of
weight $w(Y)\leq \aleph_1$. But $\{X_i^* \}_{i\in I}$ is a set
consisting of disjoint non-empty open sets of $Y$, which shows that
$w(Y)= \aleph_1$.

Finally, let $G$ be a non-empty $G_\delta$-set in
$Y$. First suppose that $\Omega\notin G$, and let $M=\bigcup_{i\in
J}X_i$ be such that $G\cap M^*\neq\emptyset$, for some countable
$J\subseteq I$. Now $G\cap M^*$ is a non-empty $G_\delta$-set in
$M^*\simeq \beta M\backslash M$. Theorem 1.2.5 of [15] states that
each non-empty $G_\delta$-set in the Stone-\v{C}ech remainder of a
locally compact $\sigma$-compact space has infinite interior.
Therefore since  $M$ is locally compact $\sigma$-compact, $G\cap
M^*$  has non-empty interior in $M^*$. But $M^*$ itself is open in
$Y$, which implies that  $\mbox {int} _Y G$ is infinite. Suppose
that $\Omega\in G$. Then we can write $G= Y\backslash
\bigcup_{n<\omega}C_n$, where each $C_n$ is a compact subset of
$\sigma X\backslash X$. Let $P=\bigcup_{i\in J}X_i$ be such that
$\bigcup_{n<\omega}C_n\subseteq \mbox {cl}_{\beta X}P$, where
$J\subseteq I$ is countable. Choose a countable $K\subseteq
I\backslash J$, and let $Q=\bigcup_{i\in K}X_i$. Then $Q^*\subseteq
Y\backslash \mbox {cl}_{\beta X}P\subseteq G$. But $Q$ contains a
copy $S$ of $\omega$ as a closed subset, and since $Q^*$ is open in
$Y$, we have $ \omega^* \simeq S^*\subseteq Q^*\subseteq \mbox {int}
_Y G$. Now Parovi\v{c}enko Theorem completes the
proof.
\end{proof}

As it is noted in  Theorem  \ref{EWG} of [9], for a zero-dimensional
locally compact  non-compact  separable  metrizable space $X$, the
answer to the  question of whether or not ${\mathcal E}_K (X)$ and
${\mathcal E}_K (\omega)$ are order-isomorphic depends on which
model of set theory is being assumed. In the following we show that
assuming the Continuum Hypothesis, if $w(X)=\aleph_1$, then
${\mathcal E}_K (X)$ and ${\mathcal E}_K (D(\aleph_1))$  are
order-isomorphic (here $D(\aleph_1)$ denotes the discrete space of
cardinality $\aleph_1$).

\begin{theorem}\label{FGEF}
{\em [CH]} Let $X$ and $Y$ be
zero-dimensional locally compact metrizable spaces of weights
$\aleph_1$. Then ${\mathcal E}_K (X)$ and ${\mathcal E}_K (Y)$  are
order-isomorphic.
\end{theorem}

\begin{proof}
By the above theorem, $S=(\sigma X\backslash X)\cup
\{\Omega\}\simeq \omega ^*$. We claim that $\Omega$ is in fact a
$P$-point of $S$. So suppose that $\Omega\in Z\in {\mathcal Z}(S)$.
Then $S\backslash Z\subseteq\sigma X\backslash X$ being a cozero-set
in $S$ is $\sigma$-compact, and therefore, assuming  the notations
of Theorem \ref{IUUG}, there exists an $M=\bigcup_{i\in J} X_i$, where
$J\subseteq I$ is countable, such that $S\backslash Z\subseteq M^*$.
Therefore $S\backslash \mbox {cl}_{\beta X}M$ is an open
neighborhood of $\Omega$ contained in $Z$. Similarly, $\Omega '$ is
a $P$-point of $T=(\sigma Y\backslash Y)\cup \{\Omega '\}\simeq
\omega ^*$. By  W. Rudin's Theorem, under  [CH], for any two
$P$-points of $\omega^*$, there is a homeomorphism of $\omega^*$
onto itself which maps one point to another (see Theorem 7.11 of
[16]). Let $f:S\rightarrow T$ be a homeomorphism such that
$f(\Omega)=\Omega '$. Now arguing as in the proof of Theorem \ref{FFFGH}, we
can show that ${\mathcal E}_K (X)$ and ${\mathcal E}_K (Y)$ are
order-isomorphic.
\end{proof}

Clearly since $\omega\subseteq D(\aleph _1)$, ${\mathcal E}_K
(D(\aleph_1))$ contains an order-isomorphic copy of ${\mathcal E}_K
(\omega)$. What is more interesting is the converse to this which is
the subject of the next result.

\begin{corollary}\label{GGF}
{\em [CH]} There is an order
isomorphism from ${\mathcal E}_K (D(\aleph_1))$ onto a subset of
${\mathcal E}_K (\omega)$. Such an order isomorphism is never onto
${\mathcal E}_K (\omega)$.
\end{corollary}

\begin{proof}
Let $X=D(\aleph _1)$. By Theorem  \ref{FGGF}, we have
$(\sigma X\backslash X)\cup \{\Omega\}\simeq \omega ^*$. Let
$f:(\sigma X\backslash X)\cup \{\Omega\}\rightarrow \omega ^*$ be a
homeomorphism. Let $Z\in \lambda_X ({\mathcal E}_K (X))$. Then by
Lemma  6.6 of [9], we have $Z\subseteq \sigma X\backslash X$, and
therefore $Z\subseteq M^*$, for some countable $M\subseteq X$. Now
$Z$ being clopen in $X^*$ is clopen in $M^*$, and therefore  it is
clopen in $(\sigma X\backslash X)\cup \{\Omega\}$. Thus $f(Z)$ is a
clopen subset of $\omega ^*$, i.e., $f(Z)\in \lambda_{\omega}
({\mathcal E}_K (\omega))$. Let $F:\lambda_X ({\mathcal E}_K
(X))\rightarrow\lambda_{\omega} ({\mathcal E}_K (\omega))$ be
defined by $F(Z)=f(Z)$. Then  $F$ is clearly  an order-isomorphism
of $\lambda_X ({\mathcal E}_K (X))$ onto its image.

To show the second part of the theorem, we note that by Theorem
\ref{IPOG}, ${\mathcal E}_K (X)$ has no minimum whereas ${\mathcal E}_K
(\omega)$ does.
\end{proof}

We summarize the above results in the following
theorem.

\begin{theorem}\label{KHG}
Let  $X$ and $Y$ be
zero-dimensional locally compact non-separable  metrizable spaces.
Then condition (1) implies the others.
\begin{itemize}
\item[\rm(1)] ${\mathcal E}_K(X)$ and ${\mathcal E}_K(Y)$ are order-isomorphic;
\item[\rm(2)] ${\mathcal Z}(\omega\sigma X\backslash X)$ and ${\mathcal Z}(\omega\sigma Y\backslash Y)$ are order-isomorphic;
\item[\rm(3)] The Boolean algebras of clopen sets ${\mathcal B}(\omega\sigma X\backslash X)$ and ${\mathcal B}(\omega\sigma Y\backslash Y)$ are order-isomorphic;
\item[\rm(4)] $\omega\sigma X\backslash X$ and $\omega\sigma Y\backslash Y$ are homeomorphic.
\end{itemize}
Furthermore, if at least one of $X$ and $Y$ has weight
greater than $2^{\aleph_0}$, then the above conditions  are
equivalent. If we assume {\em [CH]}, then the above conditions
are all equivalent.
\end{theorem}

\section{On a subset ${\mathcal E}_S(X)$ of   ${\mathcal E}(X)$}

In this section we introduce a subset  ${\mathcal E}_S(X)$ of
${\mathcal E}(X)$ and investigate its properties and its relation to
the sets  ${\mathcal E}_K(X)$  and  ${\mathcal E}(X)$ introduced
before.

\begin{definition}\label{SGGF}
For a locally compact non-separable metrizable space $X$, let
\[{\mathcal E}_S(X)=\big\{Y=X\cup \{p\}\in {\mathcal E}(X):p \mbox { has a separable neighborhood in }Y\big\}.\]
\end{definition}

\begin{theorem}\label{WHG}
For a  locally compact non-separable metrizable space $X$ we have
\[{\mathcal E}_S(X)=\big\{Y=X\cup \{p\}\in{\mathcal E}(X):p\mbox{ has a $\sigma$-compact  neighborhood in }Y\big\}.\]
\end{theorem}

\begin{proof}
Suppose that $p$ has a $\sigma$-compact neighborhood $W$ in
$Y=X\cup \{p\}\in {\mathcal E}(X)$. Then $W$ being a union of
countably many compact (and therefore separable) subsets is
separable, and so $Y\in {\mathcal E}_S(X)$.

To show the converse,
let $\{U_n\}_{n<\omega}$ be the extension trace in $X$ corresponding
to $Y=X\cup \{p\}\in {\mathcal E}_S(X)$. Then there exists a
$k<\omega$ such that $V=U_k\cup\{p\}$ is separable. Now since $U_k$
is locally compact and  separable, it is $\sigma$-compact (see 3.8.C of [7]).
\end{proof}

In the following we first characterize the elements of $ {\mathcal E}_S(X)$ in terms of their corresponding extension traces.

\begin{lemma}\label{OOHG}
Let $X$ be a locally compact
non-separable  metrizable space and let $Y=X\cup\{p\}\in {\mathcal
E}(X)$.  Then the following conditions are equivalent.
\begin{itemize}
\item[\rm(1)] $Y\in {\mathcal E}_S(X)$;
\item[\rm(2)] For every extension trace ${\mathcal U}= \{U_n\}_{n<\omega}$ in $X$ generating  $Y$, there exists a $k<\omega$ such that for every $n\geq k$, ${\mbox {\em cl}}_X U_n\backslash U_{n+1}$ is $\sigma$-compact;
\item[\rm(3)] There exists an extension trace ${\mathcal U}= \{U_n\}_{n<\omega}$ in $X$ generating  $Y$, such that for every $n<\omega$, ${\mbox {\em cl}}_X U_n\backslash U_{n+1}$ is $\sigma$-compact.
\end{itemize}
\end{lemma}

\begin{proof}
{\em (1) implies (2).} Suppose that $Y=X\cup\{p\}\in
{\mathcal E}_S(X)$, and let  $\{U_n\}_{n<\omega}$ be an extension
trace in $X$ which generates $Y$. Since $p$ has a $\sigma$-compact neighborhood in $Y$, there exist a $k<\omega$ such
 that $\mbox {cl}_X U_k\cup\{p\}$ is $\sigma$-compact, and therefore  $\mbox {cl}_X
U_n\backslash U_{n+1}$ is $\sigma$-compact, for every $n\geq k$.
That (2) implies (3), and (3) implies (1) are
trivial.
\end{proof}

The next result shows how ${\mathcal E}_S(X)$ is related to ${\mathcal E}_K(X)$.

\begin{theorem}\label{TTG}
Let $X$ be a locally compact
non-separable metrizable space. Then we have
\[{\mathcal E}_S(X)=\big\{Y\in{\mathcal E}(X):Y\geq S\mbox { for some  }S\in {\mathcal E}_K(X)\big\}.\]
\end{theorem}

\begin{proof}
Suppose that $Y\in {\mathcal E}(X)$ and let $S\in
{\mathcal E}_K(X)$ be such that $Y\geq S $. Let ${\mathcal U}= \{U_n\}_{n<\omega}$ and ${\mathcal V}= \{V_n\}_{n<\omega}$ be extension traces in $X$, corresponding to $Y=X\cup\{p\}$ and $S=X\cup\{q\}$, respectively. Since $S\in {\mathcal E}_K(X)$, there exists an $n<\omega$ such that  $\mbox {cl}_X V_n\cup\{q\}$ is compact. Since $Y\geq S$,  ${\mathcal U}$ is finer than ${\mathcal V}$, and therefore there exists a $k<\omega$ such that $U_k\subseteq V_n$. Now since
\[\mbox {cl}_X U_k=\bigcup_{i\geq n}\big(\mbox {cl}_X U_k\cap(\mbox {cl}_X V_i\backslash V_{i+1})\big)\]
$\mbox {cl}_X U_k\cup\{p\}$ is a $\sigma$-compact neighborhood of $p$ in $Y$, i.e., $Y\in {\mathcal E}_S(X)$.

For the converse, let $Y\in {\mathcal E}_S(X)$ and let ${\mathcal
U}= \{U_n\}_{n<\omega}$ be an extension traces in $X$ generating
$Y$, with $\mbox {cl}_X U_n\backslash U_{n+1}$ being
$\sigma$-compact for all $n<\omega$. Since $\mbox {cl}_X U_1=\bigcup_{n<\omega} (\mbox {cl}_X U_n\backslash U_{n+1})$, we have $\lambda (Y)\subseteq \mbox {cl}_{\beta X} U_1\subseteq \sigma X$, and thus $Y\geq S$ for some $S\in {\mathcal E}_K(X)$.
\end{proof}

The following, under [CH], describes ${\mathcal E}_S(X)$ order-theoretically as a
subset of ${\mathcal E}(X)$.

\begin{theorem}\label{TUUG}
{\em [CH]} Let $X$ be a locally compact non-separable metrizable space and let $S\in {\mathcal
E}(X)$. Then $S\in {\mathcal E}_S(X)$ if and only if
\[\big|\big\{Y\in {\mathcal E}(X):Y\geq S\big\}\big|\leq \aleph_1.\]
\end{theorem}

\begin{proof}
Suppose that $S\in {\mathcal E}(X)$ is such that
$|\{Y\in {\mathcal E}(X):Y\geq S\}|\leq \aleph_1$. Then arguing as
in the proof of Theorem  \ref{RPG} we have $\lambda (S)\subseteq \sigma
X$, and therefore $S\geq T$, for some $T\in {\mathcal E}_K(X)$.

Conversely, if $S\in {\mathcal E}_S(X)$, then there exists a $T\in
{\mathcal E}_K(X)$ such that $S\geq T$. Now the result follows, as by
the proof of Theorem  \ref{RPG}  we have $|\{Y\in {\mathcal E}(X):Y\geq
T\}|\leq \aleph_1$.
\end{proof}

Next we find the image of ${\mathcal E}_S(X)$ under $\lambda$. This will be used in the subsequent results.

\begin{theorem}\label{WG}
Let $X$ be a locally compact
non-separable metrizable space. Then we have
\[\lambda\big({\mathcal E}_S(X)\big)=\big\{Z\in {\mathcal Z}(\omega\sigma X\backslash X):\Omega\notin Z\big\}\backslash\{\emptyset\}.\]
\end{theorem}

\begin{proof}
Let $Y\in {\mathcal E}_S(X)$. Then there exists an
$S\in {\mathcal E}_K(X)$ such that $Y\geq S$, and thus $\lambda
(Y)\subseteq\lambda (S)\subseteq\sigma X$. Now since $\lambda (S)$
is clopen in $\omega\sigma X\backslash X$, we have  $\lambda
(Y)\in{\mathcal Z}(\omega\sigma X\backslash X)$.

Conversely, suppose that $\emptyset\neq Z\in{\mathcal Z}(\omega\sigma X\backslash X)$
and $\Omega\notin Z$.  Then since $Z\subseteq\sigma X\backslash X$,
we have  $Z\subseteq\lambda (S)$, for some $S\in {\mathcal E}_K(X)$.
Thus $Z=\lambda (Y)$ for some $Y\geq S$, i.e.,  $Y\in {\mathcal
E}_S(X)$.
\end{proof}

Combined with Theorem \ref{TTRPG}, the following theorem shows that whenever
 $w(X)\geq 2^{\aleph_0}$, the sets  ${\mathcal E}(X)$, ${\mathcal E}_K(X)$ and
${\mathcal E}_S(X)$ have three distinct order structures.

\begin{theorem}\label{EWG}
Let  $X$ be a locally compact
non-separable  metrizable space. Then $ {\mathcal E}_K(X) $ and $
{\mathcal E}_S(X) $ are never order-isomorphic. If moreover
$w(X)\geq 2^{\aleph_0}$, then  $ {\mathcal E}(X)$ and $ {\mathcal
E}_S(X) $ are never order-isomorphic.
\end{theorem}

\begin{proof}
The first part of the theorem follows from an
argument similar to that of Theorem \ref{TTRPG}.

To show the second part, suppose that  $\phi:{\mathcal
E}(X)\rightarrow{\mathcal E}_S(X)$ is  an order-isomorphism. By part
(4) of Theorem  \ref{YIOPY}, there exists a sequence $\{Y_n\}_{n<\omega}$
in $ {\mathcal E}(X)$ with no lower bound in $ {\mathcal E}(X)$. Let
for each $n<\omega$, $S_n\in {\mathcal E}_K(X) $ be such that
$S_n\leq \phi(Y_n)$ (see Theorem \ref{TTG}). Then by  part (5) of Theorem
\ref{YIOPY}, the sequence $\{S_n\}_{n<\omega}$ and therefore the sequence
$\{\phi(Y_n)\}_{n<\omega}$ has a  lower bound in ${\mathcal E}_K(X)
$, contradicting to our
assumptions.
\end{proof}

For a point $x$ in a space $X$, we denote by  $w(x,X)$, the smallest weight of
an open  neighborhood of $x$ in $X$. In the following lemma, under [CH],
we characterize the points of  $\sigma
X\backslash X$ in $X^*$.

\begin{lemma}\label{OTG}
{\em [CH]} Let $X$  be a locally compact non-separable metrizable space. Then the set $\sigma
X\backslash X$ consists of exactly those elements $x\in X^*$ for
which $w(x, X^*)\leq\aleph_1$.
\end{lemma}

\begin{proof}
Assume the notations of Theorem \ref{IUUG}. Let $x\in \sigma
X\backslash X$. Then $x\in M^*$, where $M=\bigcup_{i\in J} X_i$, for
some countable $J\subseteq I$. Since $M$ is separable, we have
$w(M^*)\leq\aleph_1$.

Conversely, suppose that $x\in X^*$ is such that $w(x,
X^*)\leq\aleph_1$. Suppose that $x\notin \sigma X$. Let $V$ be an
open neighborhood of $x$ in $X^*$  such that $w(\mbox
{cl}_{X^*}V)\leq\aleph_1$. Let $f\in C(X^* , \mathbf{I})$ be such
that $f(x)=0$ and $f(X^*\backslash V)\subseteq \{1\}$.  We show that
there exists a $Z\in \lambda ({\mathcal E}(X))$ such that
$Z\subseteq S=Z(f)$ and $Z\backslash\sigma X\neq\emptyset$. Since
$S\in{\mathcal Z}(X^*)$, by Lemma \ref{FFFGH} of [9], there exists a
regular sequence of open sets ${\mathcal U}=\{U_n\}_{n<\omega}$ in
$X$ such that $S=\bigcap_{n<\omega}U_n^*$. Let
\[ L=\Big\{i\in I: X_i\cap\bigcap_{n<\omega}U_n\neq\emptyset\Big\}.\]
We consider the following two cases.

{\em Case 1)} Suppose that $L$ is countable.  Let for each
$n<\omega$, $ V_n=U_n\backslash\bigcup_{i\in L}X_i$. Then since
\[\bigcap_{n<\omega}V_n=\bigcap_{n<\omega}U_n\backslash\bigcup_{i\in
L}X_i=\emptyset\]
${\mathcal V}=\{V_n\}_{n<\omega}$ is an extension
trace in $X$. Now for each $n<\omega$, we have
\[U_n^*=\Big(U_n\cap\bigcup_{i\in
L}X_i\Big)^*\cup\Big(U_n\backslash\bigcup_{i\in L}X_i\Big)^*\subseteq (\sigma
X\backslash X)\cup V_n^*\]
and therefore since $x\in
S=\bigcap_{n<\omega}U_n^*\backslash \sigma X$, we have $x\in
\bigcap_{n<\omega}V_n^*$. Let $Z=\bigcap_{n<\omega}V_n^*\in\lambda
({\mathcal E}(X))$. Then clearly $Z\subseteq S$ and
$Z\backslash\sigma X\neq\emptyset$.

{\em Case 2)} Suppose that $L$ is uncountable. Let  $\{ L_n \}_{ n
<\omega }$ be a partition of $L$ into  mutually disjoint uncountable subsets. Let for each $n<\omega$
\[V_n=U_n\cap\bigcup\{X_i:i\in L_n\cup L_{n+1}\cup
\cdots\}.\]
Then ${\mathcal V}=\{V_n\}_{n<\omega}$ is an extension
trace in $X$. We verify that for each  $n<\omega$, $\mbox
{cl}_{\beta X}V_n\backslash\sigma X\neq\emptyset$. For otherwise, if
for some $n<\omega$,  $\mbox {cl}_{\beta X}V_n\subseteq\sigma X$,
then $\mbox {cl}_{\beta X}V_n\subseteq \mbox {cl}_{\beta
X}(\bigcup_{i\in H}X_i)$, for some countable $H\subseteq I$, and
therefore $V_n\subseteq\bigcup_{i\in H}X_i$, which is a
contradiction, as each $L_n$ is  chosen to be uncountable. By
compactness of $\beta X$, we have $\bigcap_{n<\omega}(\mbox
{cl}_{\beta X}V_n\backslash\sigma X)\neq\emptyset$. Let in this case
$Z=\bigcap_{n<\omega}V_n^*\in\lambda ({\mathcal E}(X))$. Then
clearly $Z\subseteq S$ and $Z\backslash\sigma X\neq\emptyset$.

Let $A\in{\mathcal E}(X)$ be such that $Z=\lambda(A)$. By Theorem
\ref{WG}, $A\notin {\mathcal E}_S(X)$ and thus by Theorem \ref{TUUG}, $|\{Y\in
{\mathcal E}(X):Y\geq A\}|> \aleph_1$. Lemma 1\ref{SGGF}9 of [6] states
that for a $\sigma$-compact space $T$  with $w(T)\leq 2^{\aleph_0}$,
we have $|C(T)| \leq 2^{\aleph_0}$. Now applying this to $\mbox
{cl}_{X^*}V$, we obtain $|{\mathcal Z}(\mbox {cl}_{X^*}V)|\leq
\aleph_1$, which is a contradiction. This proves our lemma.
\end{proof}

\begin{theorem}\label{EYYG}
{\em [CH]} Let $X$ and $Y$ be locally
compact non-separable metrizable spaces.  If $X^*$ and $Y^*$ are
homeomorphic, then ${\mathcal E}_S (X)$ and ${\mathcal E}_S (Y)$ are
order-isomorphic.
\end{theorem}

\begin{proof}
By Lemma \ref{OTG} any homeomorphism between $X^*$ and
$Y^*$ induces a homeomorphism between $\sigma X\backslash X$ and
$\sigma Y\backslash Y$. Now the proof is completed by  a slight
modification of the last part of the proof of Theorem
\ref{FFFGH}.
\end{proof}

The following is analogous to Theorems \ref{GGH} and \ref{FFFGH}. It shows how the order structure of ${\mathcal E}_S(X)$ and the topology of $\sigma X\backslash X$ are related to each other.

\begin{theorem}\label{ERYG}
Let $X$ and $Y$ be locally
compact non-separable  metrizable spaces and let $\omega\sigma
X=\sigma X \cup\{\Omega\}$ and $\omega\sigma Y=\sigma Y \cup\{\Omega
'\}$ be the one-point compactifications of $\sigma X$ and $\sigma
Y$, respectively. If ${\mathcal E}_S(X)$ and ${\mathcal E}_S(Y)$ are
order-isomorphic, then $\omega\sigma X\backslash X$ and
$\omega\sigma Y\backslash Y$ are homeomorphic.
\end{theorem}

\begin{proof}
Let $\phi: \lambda_X ({\mathcal
E}_S(X))\rightarrow\lambda _Y({\mathcal E}_S(Y))$ be  an
order-isomorphism. We extend $\phi$ by letting
$\phi(\emptyset)=\emptyset$. We define a function $ \psi :{\mathcal
Z}(\omega\sigma X\backslash X)\rightarrow{\mathcal Z}(\omega\sigma
Y\backslash Y)$ and verify that it is an order-isomorphism.

For a $Z\in {\mathcal Z}(\omega\sigma X\backslash X)$, with
$\Omega\notin Z$, let $\psi(Z)=\phi(Z)$.

Now suppose that $Z\in {\mathcal
Z}(\omega\sigma X\backslash X)$ and $\Omega\in Z$. Then
$(\omega\sigma X\backslash X)\backslash Z$, being a cozero-set in
$\omega\sigma X\backslash X$, can be written as  $(\omega\sigma
X\backslash X)\backslash Z=\bigcup_{n<\omega}Z_n$, where for each
$n<\omega$, $ Z_n\in {\mathcal Z}(\omega\sigma X\backslash X)$ and
$\Omega\notin Z_n$, and thus by Theorem \ref{WG}, $Z_n\in\lambda_X
({\mathcal E}_S(X))$.  We claim that $ \bigcup_{n<\omega}\phi (Z_n)$
is a cozero-set in $\omega\sigma Y\backslash Y$. To show this, let
$Y=\bigoplus_{i\in J} Y_i$, with each $Y_i$ being a separable
non-compact subspace. Since for each $n<\omega$,
$\phi(Z_n)\subseteq\sigma Y\backslash Y$, there exists a countable
$L\subseteq J$ such that $\bigcup_{n<\omega}\phi (Z_n)\subseteq
(\bigcup_{i\in L}Y_i)^*=\phi (A)$, for some $A\in \lambda_X
({\mathcal E}_S(X))$. We show that $\phi(A\cap
Z)=\phi(A)\backslash\bigcup_{n<\omega}\phi (Z_n)$. Since for each
$n<\omega$, $A\cap Z\cap Z_n=\emptyset$, we have $\phi(A\cap Z)\cap
\phi(Z_n)=\emptyset$, and therefore $\phi(A\cap
Z)\subseteq\phi(A)\backslash\bigcup_{n<\omega}\phi (Z_n)$. To show
the converse, let $x\in\phi(A)\backslash\bigcup_{n<\omega}\phi
(Z_n)$. Since for each $n<\omega$, $x\notin \phi (Z_n)$, there
exists a $B\in {\mathcal Z}(\omega\sigma Y\backslash Y)$ such that
$x\in B$, and for each $n<\omega$, $B\cap\phi (Z_n)=\emptyset$. If
$x\notin\phi(A\cap Z)$, then there exists a $C\in {\mathcal
Z}(\omega\sigma Y\backslash Y)$ such that $x\in C$ and
$C\cap\phi(A\cap Z)=\emptyset$.  Consider $D=\phi(A)\cap B\cap
C\in\lambda_Y ({\mathcal E}_S(Y))$, and let $E\in\lambda_X
({\mathcal E}_S(X))$ be such that $\phi (E)=D$. Then since for each
$n<\omega$, $\phi(E)\cap \phi(Z_n)=\emptyset$, we have $E\cap
Z_n=\emptyset$, and therefore $ E\subseteq Z$. On the other hand
since $\phi(E)\subseteq\phi(A)$, we have  $E\subseteq A$ and thus
$E\subseteq A\cap Z$. Therefore, $\phi (E)\subseteq \phi (A\cap Z)$,
which implies that  $\phi(E)=\emptyset$, as $\phi(E)\subseteq C$. This
contradiction shows that $x\in\phi (A\cap Z)$, and therefore
$\phi(A\cap Z)=\phi(A)\backslash\bigcup_{n<\omega}\phi (Z_n)$. Now
since  $\phi(A)$ is clopen in $\sigma Y\backslash Y$, we have
\begin{eqnarray*}
(\omega\sigma Y\backslash Y)\backslash\bigcup_{n<\omega}\phi
(Z_n) &=& \big(\phi (A)\backslash\bigcup_{n<\omega}\phi (Z_n)\big)\cup
\big((\omega\sigma Y\backslash Y)\backslash \phi (A)\big)\\ &=&\phi (A\cap Z)\cup\big((\omega\sigma Y\backslash Y)\backslash
\phi(A)\big)\in{\mathcal Z}(\omega\sigma Y\backslash Y)
\end{eqnarray*}
and our claim is verified. In this case we define $ \psi
(Z)=(\omega\sigma Y\backslash Y)\backslash\bigcup_{n<\omega}\phi
(Z_n)$.

Next we show that $\psi$
is well-defined. So assume another representation for $Z$, i.e.,
suppose that  $Z=(\omega\sigma X\backslash
X)\backslash\bigcup_{n<\omega}S_n$, with $S_n\in \lambda_X
({\mathcal E}_S(X))\cup\{\emptyset\}$, for all $n<\omega$. Suppose
that  $\bigcup_{n<\omega}\phi (Z_n)\neq\bigcup_{n<\omega}\phi
(S_n)$. Without any loss of generality we may assume that
$\bigcup_{n<\omega}\phi (Z_n)\backslash\bigcup_{n<\omega}\phi
(S_n)\neq\emptyset$. Let $x\in \bigcup_{n<\omega}\phi
(Z_n)\backslash\bigcup_{n<\omega}\phi (S_n)$. Let $ m<\omega$ be
such that $x\in \phi(Z_m)$. Then since $x\notin
\bigcup_{n<\omega}\phi (S_n)$, there exists an $A\in {\mathcal
Z}(\omega\sigma Y\backslash Y)$ such that $x\in A$ and
$A\cap\bigcup_{n<\omega}\phi (S_n)=\emptyset$. Consider $A\cap\phi
(Z_m)\in \lambda_Y ({\mathcal E}_S(Y))$. Let $B\in \lambda_X
({\mathcal E}_S(X))$ be such that $\phi(B)=A\cap\phi(Z_m)$.  Since
$\phi (B)\subseteq A$, we have $B\cap S_n=\emptyset$, for all
$n<\omega$. But  $B\subseteq Z_m\subseteq\bigcup_{n<\omega}
Z_n=\bigcup_{n<\omega} S_n$, which implies that $B=\emptyset$, which is a
contradiction. Therefore $\bigcup_{n<\omega}\phi
(Z_n)=\bigcup_{n<\omega}\phi (S_n)$, and thus  $\psi$ is well
defined.

To prove that $\psi$ is an order-isomorphism, let $S,Z\in{\mathcal
Z}(\omega\sigma X\backslash X)$ with $S\subseteq Z$. Assume that
$S\neq\emptyset$. We consider the following three cases.

{\em Case 1)} Suppose that $\Omega\notin Z$. Then
$\psi(S)=\phi(S)\subseteq\phi(Z)=\psi(Z)$.

{\em Case 2)} Suppose that  $\Omega\notin S$ and $\Omega\in Z$.  Let
$Z=(\omega\sigma X\backslash X)\backslash\bigcup_{n<\omega}Z_n$,
with $Z_n\in \lambda_X ({\mathcal E}_S(X))\cup\{\emptyset\}$, for
all $n<\omega$.  Then since $S\subseteq Z$, for each $n<\omega$, we
have  $S\cap Z_n=\emptyset$, and therefore  $\phi (S)\cap
\phi(Z_n)=\emptyset$. We have
\[\psi(S)=\phi(S)\subseteq(\omega\sigma Y\backslash
Y)\backslash\bigcup_{n<\omega}\phi (Z_n)=\psi(Z).\]

{\em Case 3)} Suppose that  $\Omega\in S$ and let
\[Z=(\omega\sigma X\backslash X)\backslash\bigcup_{n<\omega}Z_n
\mbox{ and }S=(\omega\sigma X\backslash
X)\backslash\bigcup_{n<\omega}S_n\]
where for each $n<\omega$, $S_n, Z_n\in \lambda_X
({\mathcal E}_S(X))\cup\{\emptyset\}$. Since $S\subseteq Z$ we have
$\bigcup_{n<\omega}Z_n\subseteq\bigcup_{n<\omega}S_n$, and so
$S=(\omega\sigma X\backslash X)\backslash\bigcup_{n<\omega}(S_n\cup
Z_n)$. Therefore
\[\psi(S)=(\omega\sigma Y\backslash
Y)\backslash\bigcup_{n<\omega}\big(\phi (S_n)\cup\phi
(Z_n)\big)\subseteq(\omega\sigma Y\backslash
Y)\backslash\bigcup_{n<\omega}\phi(Z_n)=\psi(Z)\]
and thus $\psi$ is an order-homomorphism.

To show that $\psi$ is an order-isomorphism, we note that
$\phi^{-1}: \lambda_Y ({\mathcal E}_S(Y))\rightarrow\lambda
_X({\mathcal E}_S(X))$ is an order-isomorphism. Let
$\gamma:{\mathcal Z}(\omega\sigma Y\backslash Y)\rightarrow{\mathcal
Z}(\omega\sigma X\backslash X)$ be its induced order-homomorphism
defined as above. Then it is straightforward to see that
$\gamma=\psi ^{-1}$, i.e., $\psi$ is an order-isomorphism and thus
${\mathcal Z}(\omega\sigma X\backslash X)$ and ${\mathcal
Z}(\omega\sigma Y\backslash Y)$ are order-isomorphic, which implies
that $\omega\sigma X\backslash X$ and $\omega\sigma Y\backslash Y$
are homeomorphic.
\end{proof}

The next result is the converse of the above theorem under some
weight restrictions.

\begin{theorem}\label{EOYG}
Let $X$ and $Y$ be locally
compact non-separable  metrizable spaces. Suppose moreover that at
least one of $X$ and $Y$ has weight greater than $2^{\aleph_0}$. Let
$\omega\sigma X$ and $\omega\sigma Y$ be as in the above theorem.
Then if $\omega\sigma X\backslash X$ and $\omega\sigma Y\backslash
Y$ are homeomorphic, ${\mathcal E}_S(X)$ and ${\mathcal E}_S(Y)$ are
order-isomorphic.
\end{theorem}

\begin{proof}
This follows by  a slight modification of the proof of Theorem \ref{FFFGH}.
\end{proof}

We summarize the above theorems as follows.

\begin{theorem}\label{WEOYG}
Let $X$ and $Y$ be locally compact non-separable  metrizable spaces. Suppose moreover that at
least one of $X$ and $Y$ has weight greater than $2^{\aleph_0}$.
Then the following conditions are equivalent.
\begin{itemize}
\item[\rm(1)] ${\mathcal E}_S(X)$ and ${\mathcal E}_S(Y)$ are order-isomorphic;
\item[\rm(2)] ${\mathcal Z}(\omega\sigma X\backslash X)$ and ${\mathcal Z}(\omega\sigma Y\backslash Y)$ are order-isomorphic;
\item[\rm(3)] $\omega\sigma X\backslash X$  and $\omega\sigma Y\backslash Y$ are homeomorphic.
\end{itemize}
\end{theorem}

Comparing Theorems \ref{KHG} and \ref{WEOYG}, we deduce that for
zero-dimensional locally compact non-separable metrizable spaces $X$
and $Y$, such that at least one of them has weight  greater than
$2^{\aleph_0}$, ${\mathcal E}_S(X)$ and ${\mathcal E}_S(Y)$ are
order-isomorphic if and only if ${\mathcal E}_K(X)$ and ${\mathcal
E}_K(Y)$ are. It turns out that even more is true.

\begin{theorem}\label{QYG}
Let $X$ and $Y$ be
zero-dimensional locally compact non-separable  metrizable spaces
and let $f:{\mathcal E}_K(X)\rightarrow{\mathcal E}_K(Y)$ be an
order-isomorphism. Then there exists an order-isomorphism
$F:{\mathcal E}_S(X)\rightarrow{\mathcal E}_S(Y)$ such that
$F|{\mathcal E}_K(X)=f$.
\end{theorem}

\begin{proof}
Let $g=\lambda_Y f \lambda_X^{-1}:\lambda_X ({\mathcal
E}_K(X))\rightarrow\lambda_Y ({\mathcal E}_K(Y))$ and  $G:{\mathcal
B}(\omega\sigma X\backslash X)\rightarrow{\mathcal B}(\omega\sigma
Y\backslash Y)$ be as defined in the proof of Theorem  \ref{GGH}. Then as
it is shown there,  $G$ is an order-isomorphism,  and since
$\omega\sigma X\backslash X$ and $\omega\sigma Y\backslash Y$ are
zero-dimensional,  there exists a homeomorphism $\phi :\omega\sigma
X\backslash X\rightarrow\omega\sigma Y\backslash Y$ such that $\phi
(U)=G(U)$, for any $U\in{\mathcal B}(\omega\sigma X\backslash X)$.
Let $H :\lambda_X ({\mathcal E}_S(X))\rightarrow\lambda_Y ({\mathcal
E}_S(Y))$ be defined by $H (Z)=\phi(Z)$. We verify that $H$ is a
well-defined order-isomorphism.

First we note that $ \phi (\Omega)=\Omega '$. For otherwise, if $
\phi (x)=\Omega '$, for some $x\neq\Omega $, then since $x\in \sigma
X\backslash X$,  assuming  the notations of Theorem \ref{IUUG}, we have
$x\in (\bigcup _{i\in L}X_i)^*=U$, for some countable $L\subseteq
I$. Now since $U$ is clopen in $\omega\sigma X\backslash X$, we have
$\phi (U)=G(U)$, and by the way we defined $G$, $G(U)=g(U)\in
\lambda_Y ({\mathcal E}_K(Y))$. But this implies that  $\Omega '=\phi
(x)\in \phi(U)\in \lambda_Y ({\mathcal E}_K(Y))$, which is a
contradiction. Therefore $ \phi (\Omega)=\Omega '$.

Now suppose that  $Z\in\lambda_X ({\mathcal
E}_S(X))$. Then by Theorem  \ref{WG}, $Z\in {\mathcal Z}(\omega\sigma
X\backslash X)$ and $\Omega\notin Z$. Therefore $\phi (Z)$ is a
zero-set in $\omega\sigma Y\backslash Y$ such that $\Omega'\notin
\phi(Z)$. This shows that $H$  is well-defined. By the way we
defined $H$, it is clearly an order-isomorphism. Now let
$U\in\lambda_X ({\mathcal E}_K(X))$. Then since $U\in {\mathcal
B}(\omega\sigma X\backslash X)$, we have $H(U)=\phi(U)=G(U)$. But by
definition of $G$, since $\Omega\notin U$, $G(U)=g(U)$, and
therefore $H(U)=g(U)$, i.e, $H|\lambda_X ({\mathcal E}_K(X))=g$. Now
let $F=\lambda_Y^{-1}H \lambda_X:{\mathcal
E}_S(X)\rightarrow{\mathcal E}_S(Y)$. Clearly $F$ is an
order-isomorphism and by definition of $g$, for any $A\in {\mathcal
E}_K(X)$, we have $ F(A)=\lambda_Y^{-1}H \lambda_X
(A)=\lambda_Y^{-1}g \lambda_X (A)=f(A)$, i.e., $F|{\mathcal
E}_K(X)=f$ and the proof is
complete.
\end{proof}

The next result gives an order-theoretic  characterization of
${\mathcal E}_K(X)$ as a subset of ${\mathcal E}_S(X)$.

\begin{theorem}\label{WQYG}
Let $X$  be a locally compact
non-separable  metrizable spaces.  For a set ${\mathcal F}\subseteq
{\mathcal E}_S(X)$ consider the following conditions.
\begin{itemize}
\item[\rm(1)] For each $A\in{\mathcal E}_S(X)$, there exists a $B\in {\mathcal F}$ such that $B\leq A$;
\item[\rm(2)] For each $A,B\in\mathcal{F}$ such that $A<B$, there exists a $C\in\mathcal{F}$ such that $B\wedge C=A$, and $B$ and $C$ have no common upper bound in ${\mathcal E}_S(X)$.
\end{itemize}
Then the set ${\mathcal E}_K(X)$ is the largest (with
respect to set-theoretic inclusion) subset of ${\mathcal E}_S(X)$
satisfying the above conditions.
\end{theorem}

\begin{proof}
We first check that ${\mathcal E}_K(X)$ satisfies the
above conditions.  Condition (1) follows from Theorem \ref{TTG}. To show that ${\mathcal E}_K(X)$ satisfies  condition  (2), suppose  that $A,
B\in{\mathcal E}_K(X)$, with $A< B$. Let $\lambda(C)=\lambda
(A)\backslash\lambda (B)$,  for some $C\in{\mathcal E}_K(X)$.
Clearly $C\geq A$, and if $Y\in{\mathcal E}_S(X)$ is such that
$C\geq Y$ and $B\geq Y$, then $A\geq Y$. Thus $A=B\wedge C$. It is
clear that $B$ and $C$ have no common upper bound in ${\mathcal
E}_S(X)$.

Now suppose that  ${\mathcal F}\subseteq {\mathcal E}_S(X)$
satisfies conditions (1) and (2). Let $Y\in {\mathcal F}$. Then by
Theorem \ref{WG}, we have $\lambda(Y)\subseteq \sigma X\backslash X$.
Assume the notations of Theorem \ref{IUUG}. Then
$\lambda(Y)\subseteq(\bigcup_{i\in J}X_i)^*=\lambda(A)$ (properly),
for some countable $J\subseteq I$. Using condition (1), let $B\in
{\mathcal F}$ be such that $B\leq A$. Then since $Y>B$, by condition
(2), there exists a $C\in {\mathcal F}$ such that $Y\wedge C=B$ and
$Y, C$ have no common upper bound in ${\mathcal E}_S(X)$. Let $D\in
{\mathcal E}_S(X)$ be such that
$\lambda(D)=\lambda(Y)\cup\lambda(C)$. Then since $D\leq Y $ and
$D\leq C $ we have $D\leq B $. Also since $Y\geq B$ and $C\geq B$,
we have  $\lambda (B)\supseteq\lambda (D)$, and therefore $B=D$. Now
as $Y$ and $C$ have no upper bound in common $\lambda (Y)\cap\lambda
(C)=\emptyset$, and therefore as $\lambda (Y)=\lambda (A)\backslash
(\lambda (A)\cap\lambda (C))$, $\lambda (Y)$  is a clopen subset of
$X^*$. Thus $Y\in {\mathcal E}_K(X)$ and  therefore ${\mathcal
F}\subseteq {\mathcal E}_K(X)$.
\end{proof}

\begin{theorem}\label{PQYG}
Let $X$ and $Y$ be  locally compact non-separable  metrizable spaces. Then for any order-isomorphism $\phi:{\mathcal E}_S(X)\rightarrow{\mathcal E}_S(Y)$ we have $\phi({\mathcal E}_K(X))={\mathcal E}_K(Y)$.
\end{theorem}

\begin{proof}
Let ${\mathcal F}=\phi({\mathcal E}_K(X))$. Then it is easy to see that ${\mathcal F}$ satisfies the conditions of Theorem \ref{WQYG}, and thus by maximality ${\mathcal F}\subseteq {\mathcal E}_K(Y)$. The reverse inclusion holds by symmetry.
\end{proof}

The following result is analogous to Theorem \ref{FGEF}, replacing ${\mathcal E}_K(X)$ and ${\mathcal E}_K(Y)$ by ${\mathcal E}_S(X)$ and ${\mathcal E}_S(Y)$, respectively.

\begin{theorem}\label{UQYG}
{\em [CH]} Let $X$ and $Y$ be
zero-dimensional  locally compact   metrizable spaces of weights
$\aleph_1$. Then ${\mathcal E}_S(X)$ and ${\mathcal E}_S(Y)$ are
order-isomorphic.
\end{theorem}

\begin{proof}
By  Theorem \ref{FGEF}, ${\mathcal E}_K(X)$ and ${\mathcal
E}_K(Y)$ are order-isomorphic. The result now follows as by Theorem
\ref{QYG} every such order-isomorphism can be extended to an
order-isomorphism of ${\mathcal E}_S(X)$ onto ${\mathcal
E}_S(Y)$.
\end{proof}

The following example shows that zero-dimensionality cannot be
omitted from  Theorems \ref{UQYG} and \ref{FGEF}.

\begin{example}\label{OYG}
Let $X=D(\aleph_1)$ and
$Y=\bigoplus_{i<\omega_1} Y_i$, where for each $i<\omega_1$, $Y_i=
\mathbf{R}$. Suppose  that ${\mathcal E}_S(X)$  and ${\mathcal
E}_S(Y)$ are order-isomorphic, and let $\phi:{\mathcal
E}_S(X)\rightarrow{\mathcal E}_S(Y)$ denote an order-isomorphism. By
Theorem
 \ref{PQYG}, $\phi|{\mathcal E}_K(X):{\mathcal E}_K(X)\rightarrow{\mathcal
E}_K(Y)$ is also an order-isomorphism. Let $j<\omega_1$ and let
$T\in {\mathcal E}_K(X)$ be such that $\lambda(T)=Y_{j}^*$. Let
$S\in {\mathcal E}_K(X)$ be such that $\phi(S)=T$. Then since the
number of clopen subsets of $Y_{j}^*=\mathbf{R}^*$ is finite, there
are only finitely many $A\in{\mathcal E}_K(X)$ such that $A\geq S$,
which is a contradiction, as $\lambda_X(S) \subseteq D^*\simeq
\omega^*$, for some countable $D\subseteq X$, and there are
infinitely many clopen subsets  of $\lambda_X(S)$ each corresponding
to an element $A\in{\mathcal E}_K(X)$  with $A\geq
S$.
\end{example}

In Theorems \ref{RPG} and \ref{WQYG}, we characterized ${\mathcal E}_K(X)$ among the subsets of ${\mathcal E}(X)$ and ${\mathcal E}_S(X)$, respectively. In the following we give a characterization of ${\mathcal E}_S(X)$ among the subsets of ${\mathcal E}(X)$ which contain ${\mathcal E}_K(X)$.

\begin{theorem}\label{TQYG}
Let $X$ be a zero-dimensional
locally compact non-separable  metrizable space. Then the set
${\mathcal E}_S(X)$ is the smallest (with respect to set-theoretic
inclusion) subset of ${\mathcal E}(X)$ containing ${\mathcal
E}_K(X)$, such that for every upper bounded (in ${\mathcal E}(X)$)
 sequence in ${\mathcal E}_K(X)$, it contains its least upper bound (in ${\mathcal E}(X)$). 
\end{theorem}

\begin{proof}
Using Theorem  \ref{TTG}, it can be seen that the set
${\mathcal E}_S(X)$  satisfies the above requirements.

Now suppose that  ${\mathcal E}_K(X)\subseteq {\mathcal
F}\subseteq{\mathcal E}_S(X)$ satisfies the conditions of the
theorem. Let $Y\in {\mathcal E}_S(X)$. By Lemma \ref{FGRLG}, we have
$\lambda(Y)=\bigcap_{n<\omega}U_n^*$, for some extension trace
$\{U_n\}_{n<\omega}$ in $X$  consisting of clopen subsets of $X$.
Since $Y\in {\mathcal E}_S(X)$, by Theorem  \ref{TTG}, $Y\geq S$ for some
$S \in {\mathcal E}_K(X)$. Now for any $n<\omega$, $\lambda(S)\cap
U_n^*$ is a clopen subset of $X^*$, and therefore $\lambda(S)\cap
U_n^*=\lambda(Y_n)$, for some $Y_n\in {\mathcal E}_K(X)$. Clearly
 for any $n<\omega$, $\lambda(Y)\subseteq\lambda(Y_n)$, and thus
$Y\geq Y_n $. Therefore $Y=\bigvee_{n<\omega}Y_n$, and thus by
assumption $Y\in{\mathcal F}$, which shows that ${\mathcal
E}_S(X)\subseteq {\mathcal F}$.
\end{proof}

\section{Some  cardinality theorems}

In this section we obtain some theorems on the cardinality of the
sets ${\mathcal E}_K(X)$ and ${\mathcal E}(X)$. By modifying the
proofs, similar results can be obtained on the cardinality of  the
set ${\mathcal E}_S(X)$.

\begin{theorem}\label{PQYG}
Let $X$ be a locally compact
non-compact metrizable space. Then we have
\[\big|{\mathcal E}(X)\big|= 2^{w(X)}.\]
\end{theorem}

\begin{proof}
First suppose that  $X$ is separable. By Theorem \ref{OOHG}
of [9], ${\mathcal E}(X)$ and ${\mathcal Z}(X^*)\backslash
\{\emptyset\}$ are order-anti-isomorphic. Since $X$ is non-compact,
it  contains a copy $M$ of $\omega$ as a closed subset, and
therefore since $ \omega^*\simeq \mbox {cl}_{\beta X}M\backslash M
$, we may assume that $\omega^*\subseteq X^*$. But $\omega^*$ is
$z$-embedded in $X^*$ and therefore
\[\big|{\mathcal E}(X)\big|=\big|{\mathcal Z}(X^*)\big|\geq \big|{\mathcal Z}(\omega^*)\big|= 2^{\aleph_0}=2^{w(X)}.\]

Now suppose that  $X$ is non-separable and assume
the notations  of Theorem \ref{IUUG}.  Clearly $|I|=w(X)$. For each $i\in
I$, let $\{U^i_n\}_{n<\omega}$ be an extension trace in $X_i$. For
each non-empty $J\subseteq I$ and each $n<\omega$, let
$V^n_J=\bigcup_{i\in J} U^i_n$. Then it is  easy to see that
${\mathcal V}_J=\{V^n_J\}_{n<\omega}$ is an extension trace in $X$
and ${\mathcal V}_{J_1}$ and ${\mathcal V}_{J_2}$ are non-equivalent
for  $J_1\neq J_2$. Thus in this case $|{\mathcal E}(X)|\geq
|{\mathcal P}(I)|= 2^{w(X)}$.

Finally, we note that to every extension trace $\{U_n\}_{n<\omega}$
in $X$, there corresponds a sequence $\{{\mathcal B}_n\}_{n<\omega}$
of subsets of ${\mathcal B}$, where ${\mathcal B}$ is a base for $X$
of cardinality $w(X)$, in such a way that $U_n=\bigcup{\mathcal
B}_n$, for all $n<\omega$. Since the number of such sequences does
not exceed $|{\mathcal P}({\mathcal B})|^{\aleph_0}=2^{w(X)}$, it
follows that $2^{w(X)}\geq |{\mathcal E}(X)|$, and thus combined
with above, this implies that  equality
holds.
\end{proof}

By a known result of Tarski (Tarski Theorem) for any infinite set $E$, there is a
collection $\mathcal{A}$ of subsets of $E$ such that $|{\mathcal
A}|=|E|^{\aleph_0}$, $|A|=\aleph_0$ for any $A\in \mathcal{A}$ and
the intersection of any two distinct elements of $\mathcal{A}$ is
finite (see Theorem \ref{LKG} of [10]). We use this in the following
theorem.

\begin{theorem}\label{UYG}
Let  $X$ be a locally compact
non-compact metrizable space. Then we have
\[\big|{\mathcal E}_K(X)\big|\leq w(X)^{\aleph_0}.\]
Furthermore, if  $X$ is non-separable or zero-dimensional,
then equality holds.
\end{theorem}

\begin{proof}
Let $Y\in {\mathcal E}_K(X)$. Then by Lemma \ref{LYFG}, there
exists an extension trace ${\mathcal U}= \{U_n\}_{n<\omega}$ in $X$
which generates $Y$, and $\mbox {cl}_X U_n\backslash U_{n+1}$ is
compact for all $n<\omega$. Let ${\mathcal B}$ be a base in $X$ with
$|{\mathcal B}|=w(X)$, and let $U_0=X$. Then since for all
$n<\omega$, $\mbox {cl}_X U_n\backslash U_{n+1}$ is a compact subset
of the open set $ U_{n-1}\backslash \mbox {cl}_X U_{n+2}$, it
follows that there exists a $k_n<\omega$ and $C^n_1,\ldots,C^n_{k_n}\in
{\mathcal B}$ such that
\[\mbox {cl}_X U_n\backslash U_{n+1}\subseteq C^n_1\cup\cdots\cup
C^n_{k_n}\subseteq U_{n-1}\backslash \mbox {cl}_X U_{n+2}.\]
Let
\[ V_n=\bigcup_{i=1}^{k_n}C^n
_i\cup\bigcup_{i=1}^{k_{n+1}}C^{n+1} _i\cup\cdots.\]
Then clearly $V_n\subseteq U_{n-1}$. On the other hand
since $\bigcap_{n<\omega}U_n=\emptyset$, it follows from $\mbox
{cl}_X U_j\backslash U_{j+1}\subseteq \bigcup_{i=1}^{k_j}C^j _i$,
$j=n, n+1,\ldots$  that $\mbox {cl}_X U_n\subseteq V_{n}$. Now
${\mathcal V}= \{V_n\}_{n<\omega}$ is an extension trace in $X$
equivalent to ${\mathcal U}$, and therefore $Y$ is also generated by
 ${\mathcal V}$. So to each $Y\in {\mathcal E}_K(X)$, there
corresponds a sequence $\{\{C_i^n\}_{i=1}^{k_n}\}_{n<\omega}$ which
consists of finite subsets of ${\mathcal B}$. Since the number of
such sequences is not greater than  $w (X)^{\aleph_0}$, we have
$|{\mathcal E}_K(X)|\leq w(X)^{\aleph_0}$.

For the second part of the theorem we consider the following two
cases.

{\em Case 1)} Suppose that $X$ is separable and zero-dimensional. By
Proposition \ref{SGGF} of [9], $X=\bigcup _{n<\omega} C_n$, where each
$C_n$ is open with compact closure in $X$, and $\mbox {cl}_X
C_n\subseteq C_{n+1}$ for all $n<\omega$. Let $U_n=X\backslash \mbox
{cl}_X C_n$. Then clearly ${\mathcal U}=\{U_n\}_{n<\omega}$ is an
extension trace in $X$. By Lemma \ref{FGRLG} we can choose an extension
trace ${\mathcal V}=\{V_n\}_{n<\omega}$ consisting of clopen subsets
of $X$ equivalent to ${\mathcal U}$. For each $n<\omega$, there
exists a $k_n$ such that $U_{k_n}\subseteq V_n$, and so $X\backslash
V_n\subseteq X\backslash U_{k_n}=\mbox {cl}_X C_{k_n}$. Therefore
$X\backslash V_n$ and thus $V_n\backslash V_{n+1}$ is compact. Let
$D_1, D_2,\ldots$ be distinct non-empty sets of the form $
V_n\backslash V_{n+1}$. Now let $\{N_t\}_{t<2^{\aleph_0}}$ be a
partition of $\omega$ into  infinite almost disjoint subsets.  For
$t<2^{\aleph_0}$ let $N_t= \{n^t_1, n^t_2,\ldots\}$, where $n_i^t\neq
n_j^t$, for distinct $i, j<\omega$, and let
\[{\mathcal
V}_t=\{D_{n^t_k}\cup D_{n^t_{k+1}}\cup\cdots\}_{k<\omega}\]
which is
an extension trace of clopen subsets of $X$. Clearly each ${\mathcal
V}_t$ is corresponding to an elements of ${\mathcal E}_K(X)$, and
since the corresponding members of  ${\mathcal E}_K(X)$ are distinct
we have $|{\mathcal E}_K(X)|\geq 2^{\aleph_0}$.

{\em Case 2)} Suppose that $X$ is not separable and assume the
notations of Theorem \ref{IUUG}. By Tarski Theorem, there exists a
collection ${\mathcal J}$ of subsets of $I$, with $| {\mathcal
J}|=|I|^{\aleph_0}=w(X)^{\aleph_0}$  and $| J|=\aleph_0$, for every
$J\in {\mathcal J}$, such that the intersection of any two distinct
elements of ${\mathcal J}$ is finite. Let for each $J\in {\mathcal
J}$, $Y_J\in{\mathcal E}_K(X)$ be such  that
$\lambda(Y_J)=(\bigcup_{i\in J} X_i)^*$. Then $\{Y_J:J\in {\mathcal
J}\}$ is a collection of distinct elements of ${\mathcal E}_K(X)$
and therefore $|{\mathcal E}_K(X)|\geq
w(X)^{\aleph_0}$.
\end{proof}

\section{Some  questions}

Assume [CH]. Let $X$ and $Y$ be zero-dimensional locally compact
non-separable metrizable spaces. Suppose that $\phi: {\mathcal
E}(X)\rightarrow {\mathcal E}(Y)$ is an order-isomorphism. Then,
since $\phi( {\mathcal E}_K(X))$ satisfies the conditions of Theorem
\ref{RPG}, by maximality of ${\mathcal E}_K(Y)$, we have $\phi( {\mathcal
E}_K(X))\subseteq{\mathcal E}_K(Y)$, and therefore by symmetry
$\phi( {\mathcal E}_K(X))={\mathcal E}_K(Y)$. Thus $ {\mathcal
E}_K(X)$ and ${\mathcal E}_K(Y)$ are order-isomorphic. Note that in
proof of Theorem \ref{GGH}, using its notations,  we could define a
homeomorphism $f:\omega\sigma X\backslash X\rightarrow\omega\sigma
Y\backslash Y$ such that $f(U)=G(U)$, for any $U\in {\mathcal
B}(\omega\sigma X\backslash X)$. Therefore since for any countable
$J\subseteq I$, $\Omega '\notin g(Q_J)=G(Q_J)=f(Q_J)$, we have
$f(\Omega)=\Omega '$ and thus $f| \sigma X\backslash X: \sigma
X\backslash X\rightarrow \sigma Y\backslash Y$ is a homeomorphism.
In other words, having ${\mathcal E}(X)$ and ${\mathcal E}(Y)$
order-isomorphic implies that $\sigma X\backslash X$  and $\sigma
Y\backslash Y$ are homeomorphic. We don't know if the converse also
holds. More precisely

\begin{question}
Let $X$ be a locally compact
non-separable  metrizable space. Is there a subspace of $X^*$ (in
particular $X^*$ itself) whose topology determines and is determined
by the order structure of the set ${\mathcal E}(X)$?
\end{question}

\begin{question}
Let  $X$ and $Y$ be
zero-dimensional  locally compact non-separable  metrizable spaces.
 Is every order-isomorphism  $\psi: {\mathcal
E}_K(X)\rightarrow {\mathcal E}_K(Y)$ extendable to one from
${\mathcal E}(X)$ onto ${\mathcal E}(Y)$? Are at least ${\mathcal
E}(X)$ and ${\mathcal E}(Y)$ order-isomorphic?
\end{question}

It turns out that the above two questions are related in the
following way. Suppose that for every zero-dimensional  locally
compact non-separable  metrizable spaces $X$ and $Y$,  any
order-isomorphism $\phi: {\mathcal E}(X)\rightarrow {\mathcal E}(Y)
$, induces a homeomorphism $f: X^*\rightarrow Y^* $,  in such a way
that for every $T\in \lambda_X({\mathcal E}(X))$,
$f(T)=\lambda_Y\phi\lambda_X^{-1}(T)$. Let $X=D(\aleph_1)$ and
$Y=\bigoplus_{i<\omega_1}Y_i$, where for each $i<\omega_1$, $Y_i$ is
the one-point compactification of $\omega$. Then by  Theorem \ref{FGEF},
under [CH], ${\mathcal E}_K(X)$ and  ${\mathcal E}_K(Y)$ are
order-isomorphic. Suppose that   ${\mathcal E}(X)$ and ${\mathcal
E}(Y)$ are order-isomorphic, and let $\phi$ and $f$ be as defined
above. Then as the proof of Theorem \ref{LKG} shows, there exists a
non-empty zero-set $Z\in  {\mathcal
Z}(Y^*)\backslash\lambda_Y({\mathcal E}(Y))$ such that $\mbox
{int}_{c\sigma Y}(Z\backslash \sigma Y)=\emptyset$. By Lemma \ref{OTG}, we
have  $f (X^*\backslash \sigma X)=Y^*\backslash \sigma Y$ and
therefore $\mbox {int}_{c\sigma X}(f^{-1}(Z)\backslash \sigma
X)=\emptyset$. By Theorem  6.8 of [9], this implies that
$f^{-1}(Z)\in\lambda_X({\mathcal E}(X))$, which clearly contradicts
our assumptions. Therefore, in some ways the answer to the above two
questions cannot  both be positive.

\medskip

\noindent {\bf  Acknowledgments.} The author would like to express
his sincere gratitude to Professor R. Grant Woods for his invaluable
comments during this work. The author also thanks the
referee for his careful reading and his comments which improved the
paper.

\end{document}